\documentclass[12pt]{amsart}
\usepackage{amsmath}
\usepackage{amsthm}
\usepackage{amssymb}
\usepackage{amscd}
\usepackage{amsfonts}
\usepackage{amsbsy}
\usepackage{epsfig,afterpage}
\usepackage[dvips]{psfrag}
\usepackage{subfigure}
\usepackage{epsfig}
\usepackage{amsmath}
\usepackage{srcltx}
\usepackage{amsfonts}
\usepackage{amssymb}
\usepackage{psfrag}
\usepackage{verbatim}
\usepackage{graphicx}
\usepackage{amsmath}
\usepackage{amsthm}
\usepackage{amssymb}
\usepackage{amscd}
\usepackage{amsfonts}
\usepackage{amsbsy}
\usepackage{epsfig}
\usepackage[dvips]{psfrag}
\usepackage{multirow}
\usepackage{graphics}
\usepackage{epstopdf}
\usepackage{overpic}
\usepackage{float}

\usepackage{graphicx}

\usepackage{blindtext}
\usepackage{multicol}

\usepackage[colorlinks=true, linkcolor=blue, citecolor=red, urlcolor=blue]{hyperref}

\newtheorem{theorem}{Theorem}
\newtheorem{proposition}[theorem]{Proposition}

\newtheorem{lemma}[theorem]{Lemma}
\newtheorem{example}[theorem]{Example}

\newtheorem{definition}[theorem]{Definition}

\newtheorem{mtheorem}{Theorem}

\usepackage{float}
\usepackage{tikz, wrapfig}
\tikzset{node distance=3cm, auto}
\allowdisplaybreaks

\begin{document}

\title[Equivalence of planar vector fields near the infinity]
{Topological equivalence in the infinity of a planar vector field and its principal part defined through Newton polytope}

\author[T. M. Dalbelo, R. Oliveira and O. H. Perez]
{Thais Maria Dalbelo$^{1}$, Regilene Oliveira$^{2}$, \\ Otavio Henrique Perez$^{2}$}

\address{$^{1}$Federal University of S\~{a}o Carlos (UFSCar). Rodovia Washington Luís, Km 235, Zip Code 13565-905, S\~{a}o Carlos, S\~{a}o Paulo, Brazil.}

\address{$^{2}$University of S\~{a}o Paulo (USP), Institute of Mathematics and Computer Science. Avenida Trabalhador S\~{a}o Carlense, 400, Zip Code 13566-590, S\~{a}o Carlos, S\~{a}o Paulo, Brazil.}

\email{thaisdalbelo@ufscar.br}
\email{regilene@icmc.usp.br}
\email{otavio.perez@icmc.usp.br}

\thanks{ .}


\date{}
\dedicatory{}
\maketitle

\begin{abstract}
Given a planar polynomial vector field $X$ with a fixed Newton polytope $\mathcal{P}$, we prove (under some non degeneracy conditions) that the monomials associated to the upper boundary of $\mathcal{P}$ determine (under topological equivalence) the phase portrait of $X$ in a neighbourhood of boundary of the Poincaré--Lyapunov disk. This result can be seen as a version of the well known result of Berezovskaya, Brunella and Miari \cite{Berezovskaya,BrunellaMiari} for the dynamics at the infinity, We also discuss the effect of the Poincaré--Lyapunov compactification on the Newton polytope.

\end{abstract}

\section{Introduction and statement of the problem}

There is an extensive literature concerning the use of Newton polytope in the study of planar vector fields. For instance, Berezovskaya, Brunella and Miari \cite{Berezovskaya,BrunellaMiari} gave conditions that assure the existence of a topological equivalence between an analytic planar vector field and its \emph{principal part} near an isolated singularity (see \cite{AlonsoGonzalez} for an analogous result in dimension three). In \cite{Zupanovic2} Županović studied topological equivalence of vector fields and their \emph{generalized principal parts}. The Newton polytope was also used in the topological determination of singularities of positive quadratic differential forms (or pairs of foliations) \cite{GutierrezOliveiraTeixeira} and constrained differential systems \cite{PerezSilva2}. From a practical point of view, these results say how to identify the monomials of the vector fields that determine the topological behavior near the singularity.

With respect to the dynamics near the infinity, in the book \cite{Bruno} Bruno studied asymptotic behavior of integral curves of vector fields near singular points and near the infinity. In \cite{Cano} and the references therein was investigated conditions on the Newton polytope that guarantees the existence of one-parameter family of fractional power series solutions. Finally, we refer to \cite{DeminaGineValls} and the references therein on problems concerning Newton polytope and integrability of differential equations.

Although the use of Newton polytope in the study of  planar vector fields is well understood in the literature, for the best of our knowledge, there is no explicit result concerning topological determination of the dynamics in a neighbourhood of the boundary of the \emph{Poincaré--Lyapunov disk}. In other words, as far as we know, there is no explicit result that guarantees the topological determination of the dynamics of planar polynomial vector fields in a neighbourhood of the \emph{whole} infinity. This result can be interpreted as a version of the local topological determination given by Berezovskaya, Brunella and Miari at the infinity. We believe that such result will be useful in the classification of global phase portraits of planar polynomial vector fields.  Compactification of affine algebraic varieties was investigated by Khovanskii \cite{Khovanskii1,Khovanskii2}.

The contribution of this paper is twofold. Firstly, non degeneracy conditions are given in order to assure topological equivalence between a polynomial vector field $X$ and its \emph{upper principal part} in a neighbourhood of the boundary of the Poincaré--Lyapunov disk (see Definitions \ref{def-upper-degenerated} and \ref{def-eq-infinity} below). We follow the ideas of \cite{BrunellaMiari}, but our approach is strongly inspired in toroidal compactification \cite{Khovanskii1,Khovanskii2} instead. The influence of these non degeneracy conditions in the Poincaré--Lyapunov compactification process is also investigated (see Subsection \ref{subsec-influence-non-degeneracy}). The non-existence of periodic orbits near the infinity via Newton polytope is studied in Theorem \ref{main-theorem-particular-case}. In a second moment, we study the effect of the Poincaré--Lyapunov compactification in the Newton polytope. The effect of the \emph{Poincaré compactification} on the Newton polytope was studied in \cite{Berezovskaya2,Kappos}.

The paper is structured as follows. In Section \ref{sec-poincare-pl} we present the basic concepts of the paper, namely: Poincaré compactification, Poincaré--Lyapunov compactification and Newton polytope. We also present some examples of all of these objects. Strongly inspired in Toroidal Compactification, in Section \ref{sec-compact-adapted} we introduce the notion of compactification adapted to the Newton polytope, which will be crucial for the proof of our main result (Theorem \ref{main-theorem}) in Section \ref{sec-main-theo}. The Theorem \ref{main-theorem-particular-case}, which establishes  conditions on the Newton polytope for the non existence of periodic orbits near the infinity, is proved in subsection \ref{subsec-periodic-orbits}. Finally, in Section \ref{sec-plc-newton-polygon} it is studied the effect of the Poincaré--Lyapunov compactification in the Newton polytope.

\section{Poincaré compactification, Poincaré--Lyapunov compactification and Newton polytopes}\label{sec-poincare-pl}

Poincaré and Poincaré--Lyapunov compactifications are also called as homogeneous and quasi-homogeneous compactifications, respectively.

\subsection{Poincaré compactification}

The Poincaré compactification is a well-known technique utilized in the study on global dynamics of (polynomial) vector fields. For an introduction of the method, see \cite[Chapter 5]{DumortierLlibreArtes} and \cite[Chapter 3]{Perko}. 


  \label{fig-compact-construc}



Let $X$ be a planar polynomial vector field. The dynamics at infinity can be studied by using the change of coordinates \begin{multicols}{2} \noindent\begin{equation}\label{eq-poincare-x-positive}
x = v^{-1}, \ y = uv^{-1},
\end{equation}
\begin{equation}\label{eq-poincare-x-negative}
x = -v^{-1}, \ y = uv^{-1},
\end{equation}
\begin{equation}\label{eq-poincare-y-positive}
x = uv^{-1}, \ y = v^{-1},
\end{equation}
\begin{equation}\label{eq-poincare-y-negative}
x = uv^{-1}, \ y = -v^{-1}.
\end{equation}\end{multicols}

Throughout this paper, $(u,v)$ concerns variables near the infinity, whereas $(x,y)$ concerns the original system of coordinates. Observe that, in each equation above, the variables $(u,v)$ have different meanings, because each equation represents different coordinate systems. Moreover, after such a transformations, it is necessary to multiply the vector field by $v^{\delta-1}$, where $\delta = \operatorname{deg}(X)$ in order to extend the vector field to the infinity. The domain of the extended vector field after the change of coordinates \eqref{eq-poincare-x-positive}, \eqref{eq-poincare-x-negative}, \eqref{eq-poincare-y-positive} and \eqref{eq-poincare-y-negative} will be respectively denoted by $U_{1}, V_{1}, U_{2}$ and $V_{2}$.

Just as in \cite{Panazzolo}, in what follows $X$ is written in the so called logarithmic basis. More precisely, we write $X$ as a sum of homogeneous polynomial vector fields of the form
$$X(x,y) = \displaystyle\sum_{d=-1}^{\delta -1}\sum_{m+n=d} x^{m}y^{n}\Big{(}a_{m,n}x\displaystyle\frac{\partial}{\partial x} + b_{m,n}y\displaystyle\frac{\partial}{\partial y}\Big{)},$$
with $a_{m,n},b_{m,n}\in\mathbb{R}$ and $m,n\in \mathbb{Z}$ satisfying: \textbf{(1)} For $m < -1$ or $n \leq -1$, $a_{m,n} = 0$; \textbf{(2) }For $m \leq -1$ or $n < -1$, $b_{m,n} = 0$.

In the open sets $U_{1}$ and $U_{2}$, we obtain, respectively, the vector fields
\begin{equation}\label{eq-poincare-x-positive-vf}
\overline{X}_{x}^{+}(u,v) = \displaystyle\sum_{d=-1}^{\delta -1}\sum_{m+n=d} v^{\delta-(d+1)}u^{n}\Big{(}(b_{m,n}-a_{m,n})u\displaystyle\frac{\partial}{\partial u} - a_{m,n}v\displaystyle\frac{\partial}{\partial v}\Big{)};
\end{equation}
\begin{equation}\label{eq-poincare-y-positive-vf}
\overline{X}_{y}^{+}(u,v) = \displaystyle\sum_{d=-1}^{\delta -1}\sum_{m+n=d} v^{\delta-(d+1)}u^{m}\Big{(}(a_{m,n}-b_{m,n})u\displaystyle\frac{\partial}{\partial u} - b_{m,n}v\displaystyle\frac{\partial}{\partial v}\Big{)}.
\end{equation}

The vector fields \eqref{eq-poincare-x-positive-vf} and \eqref{eq-poincare-y-positive-vf} are called, respectively, the Poincaré compactification of $X$ in the \emph{$x$-positive direction} and \emph{$y$-positive direction}. With analogous computations it is possible to obtain $\overline{X}^{-}_{x}$ and $\overline{X}^{-}_{y}$, which are, respectivelly, the Poincaré compactification in the \emph{$x$-negative direction} and \emph{$y$-negative direction}. The equations \eqref{eq-poincare-x-positive-vf}, and \eqref{eq-poincare-y-positive-vf} represent, respectively, the expression of $X$ in the open sets $U_{1}$ and $U_{2}$. In each case, the set $\{v = 0\}$ represents the infinity.

The transformations \eqref{eq-poincare-x-positive}, \eqref{eq-poincare-x-negative}, \eqref{eq-poincare-y-positive} and \eqref{eq-poincare-y-negative} can be combined into a single global transformation
\begin{equation}\label{eq-poincare-global}
x = \displaystyle\frac{\cos\theta}{r}, \ y = \displaystyle\frac{\sin\theta}{r}, \ \quad \ \theta\in\mathbb{S}^{1}, r\geq 0.
\end{equation}

After such transformation, it is necessary to multiply the vector field by $r^{\delta-1}$. The set $\{r = 0\}$ represents the infinity. The phase portrait of the compactified vector field is visualized in $\mathbb{D}^{1}$, also called \emph{Poincaré disk}. The interior of $\mathbb{D}^{1}$ presents the phase portrait of the finite part, and its boundary $\mathbb{S}^{1}\subset \mathbb{D}^{1}$ plays the role of the infinity.

The construction of the vector field in the Poincaré disk in terms of the equations \eqref{eq-poincare-x-positive}, \eqref{eq-poincare-x-negative}, \eqref{eq-poincare-y-positive}, \eqref{eq-poincare-y-negative} and \eqref{eq-poincare-global} leads us to generalize such technique to the so called \emph{Poincaré--Lyapunov compactification}.

\subsection{Poincaré--Lyapunov compactification}

In the study of singularities at infinity, it is usual to deal with very degenerated equilibrium points. Thereby, sometimes is more preferable to use a generalization of the Poincaré Compactification, called \emph{Poincaré--Lyapunov Compactification} (PLC for short). The construction is very similar to the equations \eqref{eq-poincare-x-positive}, \eqref{eq-poincare-x-negative}, \eqref{eq-poincare-y-positive}, \eqref{eq-poincare-y-negative} and \eqref{eq-poincare-global}. For an introduction on such method for the planar case, we refer to \cite[Chapters 5 and 9]{DumortierLlibreArtes}.

The PLC technique was utilized  in the study of Liénard Equations (see for instance \cite{CollDumortierProhens, Dumortier, DumortierHerssens, DumortierLi, DumortierRousseau}). In \cite{LiangHuangZhao} was given all possible phase portraits in the \emph{Poincaré--Lyapunov disk} of polynomial vector fields with degree of quasihomogeneity $4$ having isolated singularities. The dynamics of the Benoît system (which is three dimensional) in the Poincaré--Lyapunov sphere was considered in \cite{LimaLlibre}. 

Given a \emph{weight vector} $\omega = (\alpha,\beta)$ of positive integers, denote by $\operatorname{Cs}\theta$ and $\operatorname{Sn}\theta$ the unique solutions of the Cauchy problem
\begin{equation}\label{eq-cauchy-problem}
\displaystyle\frac{d}{d\theta}\operatorname{Cs}\theta = -\operatorname{Sn}^{2\alpha - 1}\theta, \  \ \displaystyle\frac{d}{d\theta}\operatorname{Sn}\theta = \operatorname{Cs}^{2\beta - 1}\theta, \ \quad \ \operatorname{Cs}0 = 1, \operatorname{Sn}0 = 0;
\end{equation}
which are analytic and periodic, whose period $T$ is given by
\begin{equation}\label{eq-period-lyapunov}
T = \displaystyle\frac{2\alpha^{\frac{1-2\alpha}{2\alpha}}}{\beta^{\frac{1}{2\alpha}}} \int_{0}^{1}(1-t)^{\frac{1-2\alpha}{2\alpha}}t^{\frac{1-2\beta}{2\beta}}dt.
\end{equation}

Such pair of functions also satisfy the equation $$\beta\operatorname{Sn}^{2\alpha}\theta + \alpha\operatorname{Cs}^{2\beta}\theta = \alpha.$$

Consider the Equation \eqref{eq-cauchy-problem}. Given a planar vector field $X$, the analytic vector field $\overline{X}$ obtained after the transformations
\begin{equation}\label{eq-cs-sn}
x = \displaystyle\frac{\operatorname{Cs}\theta}{r^{\alpha}}, \ \quad \ y = \displaystyle\frac{\operatorname{Sn}\theta}{r^{\beta}}, \ \quad \ \theta\in\mathbb{S}^{1}, r\geq 0;
\end{equation}
and by an suitable multiplication of a power of $r$ is called \emph{Poincaré--Lyapunov compactification of $X$} (PLC for short) and it is defined in the \emph{Poincaré--Lyapunov disk $\mathbb{D}_{(\alpha,\beta)}$} (PL-disk for short).
 
Just as in the homogeneous compactification, it is often preferable to deal with directional charts. In the charts $U_{1}, V_{1}, U_{2}$ and $V_{2}$, the dynamics at infinity can be studied by using respectively the formulas \begin{multicols}{2} \noindent\begin{equation}\label{eq-plc-x-positive}
x = v^{-\alpha}, \ y = uv^{-\beta},
\end{equation}
\begin{equation}\label{eq-plc-x-negative}
x = -v^{-\alpha}, \ y = uv^{-\beta},
\end{equation}
\begin{equation}\label{eq-plc-y-positive}
x = uv^{-\alpha}, \ y = v^{-\beta},
\end{equation}
\begin{equation}\label{eq-plc-y-negative}
x = uv^{-\alpha}, \ y = -v^{-\beta}.
\end{equation}\end{multicols}

Observe that in each equation above the variables $(u,v)$ have different meanings. Moreover, after such transformations, it is necessary to multiply the vector field by a suitable power of $v$.

The next step is to find the expressions of the compactified vector field $\overline{X}$ on each chart. Firstly, we write the vector field $X$ as a sum of quasi-homogeneous polynomial vector fields in the logarithmic basis
\begin{equation}\label{eq-quasi-homogeneous-vf}
X(x,y) = \displaystyle\sum_{d=-1}^{\delta -1}\sum_{\alpha m+\beta n=d} x^{m}y^{n}\Big{(}a_{m,n}x\displaystyle\frac{\partial}{\partial x} + b_{m,n}y\displaystyle\frac{\partial}{\partial y}\Big{)}.
\end{equation}

We say that equation \eqref{eq-quasi-homogeneous-vf} is a \emph{$(\alpha,\beta)$-decomposition of the planar vector field $X$}. Moveover, the vector field
$$X^{(\alpha,\beta)}_{d}(x,y) = \sum_{\alpha m+\beta n=d} x^{m}y^{n}\Big{(}a_{m,n}x\displaystyle\frac{\partial}{\partial x} + b_{m,n}y\displaystyle\frac{\partial}{\partial y}\Big{)}$$
is called \emph{$d$-level of the $(\alpha,\beta)$-decomposition of $X$}.

In the open sets $U_{1}$ and $U_{2}$, we obtain, respectively, the vector fields
\begin{equation}\label{eq-plc-x-positive-vf}
\overline{X}_{x}^{+}(u,v) =  \displaystyle\sum_{d=-1}^{\delta -1}\sum_{\alpha m+\beta n=d} v^{\delta-(d+1)}u^{n}\Big{(}(b_{m,n}-\frac{\beta}{\alpha}a_{m,n})u\displaystyle\frac{\partial}{\partial u} - \frac{1}{\alpha}a_{m,n}v\displaystyle\frac{\partial}{\partial v}\Big{)};
\end{equation}
\begin{equation}\label{eq-plc-y-positive-vf}
\overline{X}_{y}^{+}(u,v) =  \displaystyle\sum_{d=-1}^{\delta -1}\sum_{\alpha m+\beta n=d} v^{\delta-(d+1)}u^{m}\Big{(}(a_{m,n}-\frac{\alpha}{\beta}b_{m,n})u\displaystyle\frac{\partial}{\partial u} - \frac{1}{\beta}b_{m,n}v\displaystyle\frac{\partial}{\partial v}\Big{)}.
\end{equation}

The equations \eqref{eq-plc-x-positive-vf} and \eqref{eq-plc-y-positive-vf} will be called, respectively, Poincaré--Lyapunov compactification of $X$ in the \emph{$x$-positive direction} and \emph{$y$-positive direction}. With straightforward computations it is possible to obtain $\overline{X}^{-}_{x}$ and $\overline{X}^{-}_{y}$ as well, which are, respectivelly, the Poincaré--Lyapunov compactification in the \emph{$x$-negative direction} and \emph{$y$-negative direction}. The equations \eqref{eq-plc-x-positive-vf}, and \eqref{eq-plc-y-positive-vf} represent, respectively, the expression of $\overline{X}$ in the open sets $U_{1}$ and $U_{2}$. In each case, the set $\{v = 0\}$ represents the infinity. Moreover, setting $\alpha = \beta = 1$ in equations \eqref{eq-plc-x-positive-vf}, and \eqref{eq-plc-y-positive-vf}, we obtain \eqref{eq-poincare-x-positive-vf}, and \eqref{eq-poincare-y-positive-vf}.

At this point a natural question arises: \emph{how do we choose the weight vector $\omega$?} The suitable choice of $\omega$ for the PLC process is discussed in Section \ref{sec-plc-newton-polygon}, and we utilize the well known \emph{Newton polytope}.

\subsection{Newton polytope and its lower and upper boundaries}\label{subsec-newton-polytope}

In this subsection we recall the construction of the Newton polytope for planar vector fields, where can also be found in \cite{Berezovskaya2, Kappos}. Given the vector field \eqref{eq-quasi-homogeneous-vf}, we associate the monomials $a_{m,n}x^{m}y^{n}$ and $b_{m,n}x^{m}y^{n}$ with nonzero coefficients to a point $(m,n)$ in the plane of powers. Observe that each point $(m,n)$ is contained in a line of the form $\{\alpha m + \beta n = d\}$.

The \emph{support $\mathcal{S}_{X}$ of $X$} is the set
$$\mathcal{S}_{X} = \{(m,n)\in \mathbb{Z}^{2}; a_{m,n}^{2} + b_{m,n}^{2} \neq 0\}.$$

Note that the support $\mathcal{S}_{X}$ is finite due to the fact that $X$ is polynomial. The \emph{Newton polytope $\mathcal{P}_{X}$} associated to the polynomial vector field $X$ given by \eqref{eq-quasi-homogeneous-vf} is the convex hull of the support $\mathcal{S}_{X}$. Observe that the Newton polytope strongly depends on the coordinate system adopted. Moreover, when the vector field $X$ considered is well understood, for simplicity sake we denote $\mathcal{S}_{X}$ by $\mathcal{S}$ and $\mathcal{P}_{X}$ by $\mathcal{P}$. Finally, note that if the interior of $\mathcal{P}$ is empty then $\mathcal{P}$ is just a segment.

The non-smooth points of the boundary $\partial\mathcal{P}$ of the Newton polytope will be called \emph{vertices} of $\mathcal{P}$. They are enumerated in the counter clockwise sense $p_{0},\dots,p_{k}\in\partial\mathcal{P}$ in such a way that the vertex $p_{0}$ is the first point of $\mathcal{S}$ according to the Lexicographical order. The boundary $\partial\mathcal{P}$ of the Newton polytope is the union of a finite number of compact segments, and they also will be enumerated in the counter clockwise sense $\gamma_{1}, ..., \gamma_{k+1}$ in such a way that $\gamma_{l}$ contains $p_{l-1}$ and $p_{l}$, and $\gamma_{k+1}$ contains $p_{k}$ and $p_{0}$ (see Figure \ref{fig-boundaries-polygon}). Finally, we introduce the notions of \emph{lower} and \emph{upper boundaries} of the Newton polytope $\mathcal{P}$. Denote the convex envelope of $\mathcal{S} + (\mathbb{R}_{\geq0})^{2}$ by $\mathcal{P}_{+}$ and its boundary by $\partial\mathcal{P}_{+}$.




\begin{definition}
Let $\mathcal{P}$ be a Newton polytope with non empty interior. The \emph{lower boundary of $\mathcal{P}$} is the union of the compact segments of $\partial\mathcal{P}_{+}$ and it will be denoted by $\mathcal{P}^{L}$. The \emph{upper boundary of $\mathcal{P}$} is the closure of $\partial\mathcal{P} \setminus \partial\mathcal{P}_{+}$ and it will be denoted by $\mathcal{P}^{U}$. If $\mathcal{P}$ has empty interior, then $\mathcal{P} = \mathcal{P}^{L} = \mathcal{P}^{U}$. See Figure \ref{fig-boundaries-polygon}.
\end{definition}

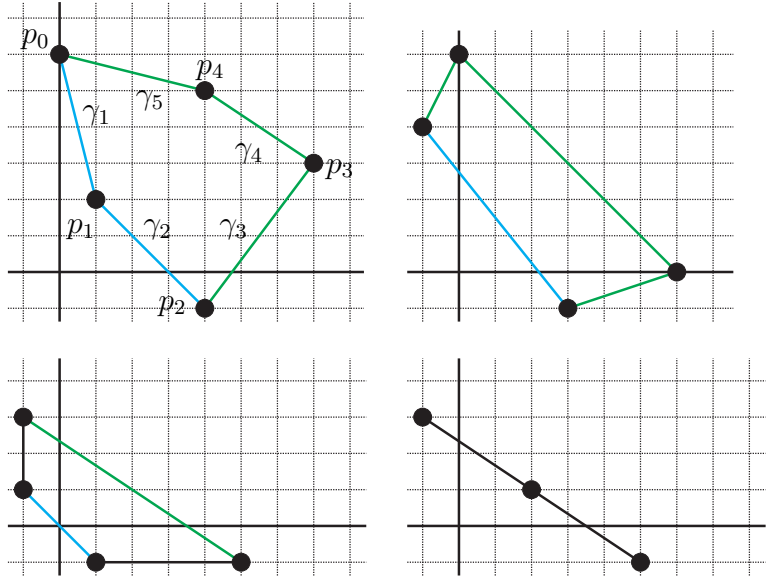
\begin{figure}[h]\center{
\begin{overpic}[width=0.8\textwidth]{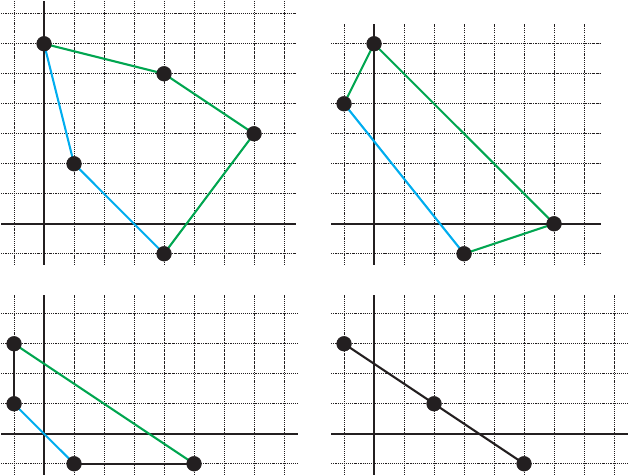}
    \put(2,70){{$p_{0}$}}
    \put(10,60){{$\gamma_{1}$}}
    \put(8,45){{$p_{1}$}}
    \put(18,45){{$\gamma_{2}$}}
    \put(20,35){{$p_{2}$}}
    \put(28,45){{$\gamma_{3}$}}
    \put(42,53){{$p_{3}$}}
    \put(30,55){{$\gamma_{4}$}}
    \put(25,66){{$p_{4}$}}
    \put(17,62){{$\gamma_{5}$}}
\end{overpic}}
\caption{\footnotesize{Examples of lower (highlighted in blue) and upper (highlighted in green) boundaries of a Newton polytope $\mathcal{P}$.}}
\label{fig-boundaries-polygon}
\end{figure}

\subsection{An example}\label{subsec-example} This subsection aims to exemplify the definitions and tools presented so far. Consider the polynomial vector field
\begin{equation*}
X(x,y) = (y^{3} - x^{3}y)\displaystyle\frac{\partial}{\partial x} + (-x^{3} + xy^{3})\displaystyle\frac{\partial}{\partial y}.   
\end{equation*}

The support is the set $\mathcal{S}_{X} = \{(-1,3),(2,1),(1,2),(3,-1)\}$ and the Newton polytope $\mathcal{P}_{X}$ is represented in the Figure \ref{fig-exe-plc}(a). Applying a Poincaré--Lyapunov compactification with weight $\omega = (1,2)$, one obtains the following vector fields in the positive $x$ and $y$ direction, respectively:
\begin{equation*}
\overline{X}^{+}_{x}(u,v) = \Big{(}u^{3} - 2u^{4} + 2u^{2}v - v^{4}\Big{)}\displaystyle\frac{\partial}{\partial u} + uv(v - u^{2}) \displaystyle\frac{\partial}{\partial v}; 
\end{equation*}
\begin{equation*}
\overline{X}^{+}_{y}(u,v) = \Big{(}1 - \frac{u^{2}}{2} - u^{3}v + \frac{u^{4} v^{4}}{2}\Big{)}\displaystyle\frac{\partial}{\partial u} + \frac{uv}{2}(u^{2}v^{4} - 1) \displaystyle\frac{\partial}{\partial v};    
\end{equation*}
whose Newton polytopes $\mathcal{P}_{\overline{X}^{+}_{x}}$ and $\mathcal{P}_{\overline{X}^{+}_{y}}$ are represented in Figure \ref{fig-exe-plc}(a). Now, consider the vector field
\begin{equation*}
Y(x,y) = \Big{(}y^{3} - x^{3}y + A(x,y)\Big{)}\displaystyle\frac{\partial}{\partial x} + \Big{(}-x^{3} + xy^{3} + B(x,y)\Big{)}\displaystyle\frac{\partial}{\partial y};   
\end{equation*}
in which $A, B$ are of the form
$$A(x,y) = a_{1}xy^{2} + a_{2}x^{3} + a_{3}x^{2}y; \quad B(x,y) = b_{1}y^{3} + b_{2}x^{2}y + b_{3}xy^{2}.$$

One more time we use the PLC technique with weight $\omega = (1,2)$, and the following vector fields are obtained in the positive $x$ and $y$ direction, respectively:
\begin{equation*}
\begin{split}
\overline{Y}^{+}_{x}(u,v) & = \Big{(}u^{3} - 2u^{4} + 2u^{2}v - v^{4} - 2uv^{6}A\Big{(}\frac{1}{v}, \frac{u}{v^{2}}\Big{)} + v^{7}B\Big{(}\frac{1}{v}, \frac{u}{v^{2}}\Big{)}\Big{)}\displaystyle\frac{\partial}{\partial u}  \\
& + \Big{(}uv^{2} - u^{3}v  - v^{7}A\Big{(}\frac{1}{v}, \frac{u}{v^{2}}\Big{)}\Big{)} \displaystyle\frac{\partial}{\partial v}; 
\end{split}
\end{equation*}
\begin{equation*}
\begin{split}
\overline{Y}^{+}_{y}(u,v) & = \Big{(}1 - \frac{u^{2}}{2} - u^{3}v + \frac{u^{4} v^{4}}{2} + v^{6}A\Big{(}\frac{u}{v}, \frac{1}{v^{2}}\Big{)} - \frac{uv^{7}}{2}B\Big{(}\frac{u}{v}, \frac{1}{v^{2}}\Big{)}\Big{)}\displaystyle\frac{\partial}{\partial u}  \\
& + \Big{(}\frac{u^{3}v^{5}}{2} - \frac{uv}{2} - \frac{v^{8}}{2}B\Big{(}\frac{u}{v}, \frac{1}{v^{2}}\Big{)}\Big{)} \displaystyle\frac{\partial}{\partial v}. 
\end{split}
\end{equation*}

Although $A$ and $B$ are homogeneous polynomials of degree three, it can be verified that $\mathcal{P}^{U}_{X} = \mathcal{P}^{U}_{Y}$, $\mathcal{P}^{L}_{\overline{X}^{+}_{x}} = \mathcal{P}^{L}_{\overline{Y}^{+}_{x}}$ and $\mathcal{P}^{L}_{\overline{X}^{+}_{y}} = \mathcal{P}^{L}_{\overline{Y}^{+}_{y}}$ (see Figure \ref{fig-exe-plc}b). This example suggests that monomials associated to the upper boundary of the polytope determine the topological behavior of the singularities at the infinity.

\begin{figure}[h]\center{
\begin{overpic}[width=0.85\textwidth]{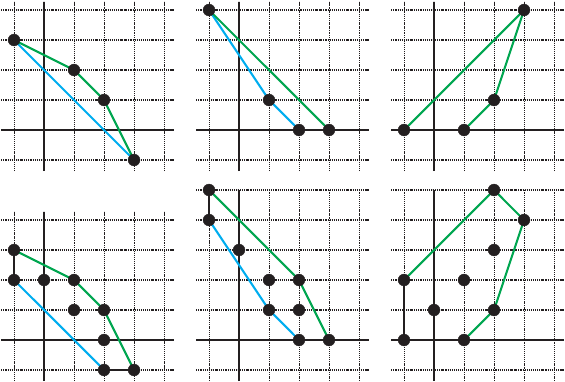}
    \put(-6,50){(a)}
    \put(-6,15){(b)}
\end{overpic}}
\caption{\footnotesize{Newton polytopes of the vector fields studied in the example of the subsection \ref{subsec-example}. Figure (a) from the left to the right: Newton polytopes $\mathcal{P}_{X}$, $\mathcal{P}_{\overline{X}^{+}_{x}}$ and $\mathcal{P}_{\overline{X}^{+}_{y}}$. Figure (b) from the left to the right: Newton polytopes $\mathcal{P}_{Y}$, $\mathcal{P}_{\overline{Y}^{+}_{x}}$ and $\mathcal{P}_{\overline{Y}^{+}_{y}}$.The lower and upper boundaries are highlighted in blue and green, respectively.}}
\label{fig-exe-plc}
\end{figure}

\section{Compactification adapted to the Newton polytope}\label{sec-compact-adapted}

Let $\mathcal{P}$ be a Newton polytope with non empty interior and let $\gamma\subset\partial\mathcal{P}$ be a segment that is contained in a line of the form $\{\mu m + \nu n = D\}$. The vector $(\mu,\nu)\in\mathbb{Z}^{2}$ is an \emph{inward normal vector of $\gamma$} (see Figure \ref{fig-inward-normal-vectors}) if it satisfies one of the following:
\begin{description}
    \item[(a)] If $\gamma\subset\mathcal{P}^{L}$, then $D$ is positive in the equation $\mu m + \nu n = D$.
    \item[(b)] If $\gamma\subset\mathcal{P}^{U}$, then $D$ is negative in the equation $\mu m + \nu n = D$.
\end{description}

\begin{figure}[h]\center{
\begin{overpic}[width=0.40\textwidth]{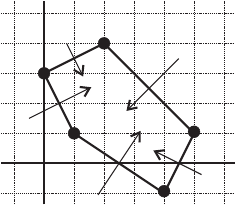}
\end{overpic}}
\caption{\footnotesize{Inward normal vectors of $\mathcal{P}$. Geometrically, when $(\mu,\nu)$ is positioned on $\gamma\subset\partial\mathcal{P}$, it points toward the interior of $\mathcal{P}$.}}
\label{fig-inward-normal-vectors}
\end{figure}

\begin{definition} The collection $\Sigma^{*}_{\mathcal{P}} = \{(\mu_{i},\nu_{i})\}_{i=1}^{r}$ of the inward normal vectors of $\mathcal{P}^{U}$ with $\gcd(\mu_{i},\nu_{i}) = 1$ is called \emph{skeleton of a simple fan}. A \emph{simple fan} is the collection of vectors $\Sigma_{\mathcal{P}} = \{\xi_{j}\}_{j = 0}^{s}$, with $\xi_{j} = (\alpha_{j},\beta_{j})\in\mathbb{Z}^{2}$, and such collection satisfies:
\begin{description}
    \item[(i)] $\Sigma^{*}_{\mathcal{P}}\subset \Sigma_{\mathcal{P}}$.
    \item[(ii)] $\xi_{0} = (0,1)$ and $\xi_{s} = (1,0)$.
    \item[(iii)] $\gcd(\alpha_{j},\beta_{j}) = 1$ and $\det\left(
  \begin{array}{cc}
    \alpha_{j-1} & \beta_{j-1} \\
    \alpha_{j} & \beta_{j} \\
  \end{array}
\right) = 1$.
    \item[(iv)] Two consecutive vectors $\xi_{j-1},\xi_{j}$ are not normal to two consecutive segments of the upper boundary of $\mathcal{P}$. 
    \item[(v)] $\Sigma_{\mathcal{P}}$ is \emph{minimal} in the sense that there is no collection satisfying conditions (i) - (iv) containing a smaller number of vectors.
\end{description}
\end{definition}

It is known from Toric Geometry \cite{Brasselet,CoxLittleSchenck} that a skeleton $\Sigma^{*}_{\mathcal{P}}$ always can be augmented in order to obtain a fan $\Sigma_{\mathcal{P}}$ satisfying conditions (i) - (iv). Condition (iii) assures that the set $B_{j} = \{\xi_{j-1},\xi_{j}\}$ of consecutive vectors of the fan $\Sigma$ will be is a basis of $\mathbb{Z}^{2}$. Observe that there is no vector $\xi_{j}\in\Sigma_{\mathcal{P}}$ that is contained in the first quadrant, for $j = 1,\dots, s-1$. For $j = 2,\dots,s-1$, each $B_{j}$ defines a map $\psi_{j}$ given by
$$
\begin{array}{cccccccc}
\psi_{j}: & (\mathbb{R}^{*})^{2} & \rightarrow &  (\mathbb{R}^{*})^{2} & ; & \psi_{j}(u,v) & = & (u^{\alpha_{j-1}}v^{\alpha_{j}}, \ u^{\beta_{j-1}}v^{\beta_{j}}); \\
\psi_{j}^{-1}: & (\mathbb{R}^{*})^{2} & \rightarrow &  (\mathbb{R}^{*})^{2} & ; & \psi_{j}^{-1}(x,y) & = & (x^{\beta_{j}}y^{-\alpha_{j}}, \ x^{-\beta_{j-1}}y^{\alpha_{j-1}}); \\
\end{array}
$$
in which $\mathbb{R}^{*} = \mathbb{R}\backslash\{0\}$. The set $B_{0} = \{\xi_{s},\xi_{0}\} = \{(1,0);(0,1)\}$ induces the identity map. For $j = 1$ and $j = s$ we obtain
$$
\begin{array}{cccccccc}
\psi_{1}: & \mathbb{R}\times\mathbb{R}^{*} & \rightarrow &  \mathbb{R}^{*} \times \mathbb{R} & ; & \psi_{1}(u,v) & = & (v^{-1}, \ uv^{\beta_{1}}); \\
\psi_{1}^{-1}: & \mathbb{R}^{*}\times\mathbb{R} & \rightarrow &  \mathbb{R}\times\mathbb{R}^{*} & ; & \psi_{1}^{-1}(x,y) & = & (x^{\beta_{1}}y, \ x^{-1});
\end{array}
$$

$$
\begin{array}{cccccccc}
\psi_{s}: & \mathbb{R}^{*}\times\mathbb{R} & \rightarrow &  \mathbb{R}\times\mathbb{R}^{*} & ; & \psi_{s}(u,v) & = & (u^{\alpha_{s-1}}v, \ u^{-1}); \\
\psi_{s}^{-1}: & \mathbb{R}\times\mathbb{R}^{*} & \rightarrow &  \mathbb{R}^{*} \times \mathbb{R} & ; & \psi_{s}^{-1}(x,y) & = & (y^{-1}, \ xy^{\alpha_{s-1}}).
\end{array}
$$

In what follows is presented the definition of compactification of the phase space adapted to the Newton polytope, which is an adaptation to our context of the definition presented in \cite[Section 1.3]{Khovanskii1}. Afterwards it will be defined compactification of a polynomial vector field adapted to the Newton polytope.

\begin{definition}
A \emph{compactification of $(\mathbb{R}^{*})^{2}$ adapted to $\mathcal{P}$} is a pair $(\mathcal{M},\Psi)$ satisfying the following conditions:
\begin{description}
    \item[(a)] $\mathcal{M}$ is a smooth and compact 2-dimensional manifold containing $(\mathbb{R}^{*})^{2}$.
    \item[(b)] $\Psi:\mathcal{M}\rightarrow\mathbb{R}^{2}$ is a map locally given by $\overline{\psi}_{j}:U_{j}\subset\mathcal{M}\rightarrow\mathbb{R}^{2}$, which is the extension of $\psi_{j}$ to a set $\mathcal{I}$ that is the union of 1-dimensional smooth manifolds in general position.
    \item[(c)] The collection of sets $U_{j}$ covers $\mathcal{M}$.
\end{description}
\end{definition}

Local coordinates in $\mathcal{M}$ and $\mathbb{R}^{2}$ are denoted by the pair $(u,v)$ and $(x,y)$, respectively. Observe that each $\psi_{j}$ is a diffeomorphism. Moreover, the sets $\{uv = 0\}\subset U_{j}$ (for $j = 2,\dots,s-1)$, $\{v = 0\}\subset U_{1}$ and $\{u = 0\}\subset U_{s}$ play the role of the infinity of $\mathbb{R}^{2}$.

It is well known from Toric Geometry (see \cite[Section 3.4]{Brasselet} and \cite[Section 3.1]{CoxLittleSchenck}) that if the manifold $\mathcal{M}$ is constructed over $\mathbb{C}$, then it will be smooth and compact (and it has complex dimension 2). In the real case the manifold $\mathcal{M}$ will also be smooth and compact. Roughly speaking, the manifold $\mathcal{M}$ contructed over $\mathbb{R}$ can be seen as a ``restriction'' of the toric variety constructed over $\mathbb{C}$ (see \cite{Sottile} for further details).

\begin{definition}\label{def-compact}
A \emph{compactification of $X$ adapted to $\mathcal{P}$} is the vector field $\overline{X}:\mathcal{M}\rightarrow T\mathcal{M}$ defined as the pullback of $X$ by $\Psi$, that is, $\overline{X} = \Psi_{*}X$. See Figure \ref{fig-diag-def-compact}.
\end{definition}

\begin{figure}[h!]
\begin{flushright}
\begin{center}
\begin{tikzpicture}
\node (A) {$T\mathcal{M}$};
\node (B) [right of=A] {$T\mathbb{R}^{2}$};
\node (C) [below of=A] {$\mathcal{M}$};
\node (D) [below of=B] {$\mathbb{R}^{2}$};
\large\draw[->] (C) to node {\mbox{{\footnotesize $\overline{X}$}}} (A);
\large\draw[->] (D) to node {\mbox{{\footnotesize $X$}}} (B);
\large\draw[->] (A) to node {\mbox{{\footnotesize $\Psi^{*}$}}} (B);
\large\draw[->] (C) to node {\mbox{{\footnotesize $\Psi$}}} (D);
\end{tikzpicture}
\end{center}
\end{flushright}
\caption{\footnotesize{Cummutative diagram of Definition \ref{def-compact}.}}
\label{fig-diag-def-compact}
\end{figure}

If one works with local charts, for each $j = 1,\dots,s$, consider the vector field $X_{j} = \psi_{j}^{*} X$, that is, $X_{j}$ is the pushforward of $X$ by $\psi_{j}$. Observe that $\psi_{j}$ is a diffeomorphism outside $\{uv = 0\}$ and $X_{j}$ can be analytically extended to such set. Denote such extension by $\overline{X}_{j}$. The domain of $\overline{X}_{j}$ is the set $U_{j}$ and $\{uv = 0\}$ plays the role of the infinity. In particular, for $\overline{X}_{1}$ the infinity is represented by $\{v = 0\}$ and for $\overline{X}_{s}$ the infinity is represented for $\{u = 0\}$.

\begin{example}
Suppose that the upper boundary $\mathcal{P}^{U}$ of a given polynomial vector field $X$ has only one segment, whose inward normal vector is $(-1,-1)$. Then the simple fan associated to $\mathcal{P}$ is the collection $\Sigma_{\mathcal{P}} = \{(0,1),(-1,-1), (0,1)\}$. Such collection induces the coordinates changes $\psi_{1}$ and $\psi_{2}$ given respectively by
$$x = \frac{1}{v}, y = \frac{u}{v}, \quad \text{and} \quad x = \frac{v}{u}, y = \frac{1}{u};$$
which is related to the classical Poincaré compactification (compare with the equations \eqref{eq-poincare-x-positive} and \eqref{eq-poincare-y-positive}). In the first case, the infinity is given by $\{v = 0\}$, whereas in the second case the infinity is the set $\{u = 0\}$. The manifold induced by $\psi_{1}$ and $\psi_{2}$ is the real projective plane $\mathbb{RP}^{2}$.
\end{example}

\begin{example}
Consider once again the vector field $X$ of the example presented in the subsection \ref{subsec-example}. The skeleton of a simple fan is the set $\Sigma^{*}_{\mathcal{P}} = \{(-2,-1); (-1,-1);(-1,-2)\}$, whereas the simple fan is 
\begin{equation*}
\begin{split}
\Sigma_{\mathcal{P}} = & \Big{\{}(0,1); (-1,0); (-2,-1); (-3,-2); (-1,-1);(-2,-3);(-1,-2); \\
& (0,-1);(1,0)\Big{\}}.
\end{split}
\end{equation*}
\end{example}

\section{Proof of the main Theorem}\label{sec-main-theo}

This section is devoted to prove Theorem \ref{main-theorem}. Firstly, it is presented the notions of \emph{upper principal part} of a planar polynomial vector field, and the so called \emph{non degenerated upper principal part}, which is a generic condition (see Proposition \ref{prop-generic}). The notion of topological equivalence in the infinity is given in Definition \ref{def-eq-infinity}. We also present a simpler proof of our main Theorem for the case in which the upper boundary $\mathcal{P}^{U}$ is just only one segment (see Theorem \ref{main-theorem-particular-case}).

\begin{definition}\label{def-upper-degenerated}
Let $X$ be a polynomial vector field. The \emph{upper principal part of $X$} is the vector field
$$X_{\Delta}^{U}(x,y) = \sum_{(m,n)\in\mathcal{P}^{U}} x^{m}y^{n}\Big{(}a_{m,n}x\displaystyle\frac{\partial}{\partial x} + b_{m,n}y\displaystyle\frac{\partial}{\partial y}\Big{)};$$
in which $\mathcal{P}^{U}$ denotes the upper boundary of $\mathcal{P}$. The upper principal part $X_{\Delta}^{U}$ is \emph{non degenerated} if any quasi homogeneous component 
$$X_{\gamma_{j}}(x,y) = \sum_{(m,n)\in\gamma_{j}} x^{m}y^{n}\Big{(}a_{m,n}x\displaystyle\frac{\partial}{\partial x} + b_{m,n}y\displaystyle\frac{\partial}{\partial y}\Big{)};$$
has no singularities in $(\mathbb{R}^{*})^{2}$, in which $\gamma_{j}\subset\mathcal{P}^{U}$. 
\end{definition}

The influence of this non degeneracy condition in the PLC process is discussed in the subsection \ref{subsec-influence-non-degeneracy}. It is important to observe that such condition is \emph{generic}. More precisely, we have the following proposition.

\begin{proposition}\label{prop-generic}
Denote the set of all polynomial vector fields defined in $\mathbb{R}^{2}$ by $\mathfrak{X}$, and the set of all polynomial vector fields with non degenerated upper principal part by $\mathfrak{U}$. Then the set $\mathfrak{U}$ is open and dense in $\mathfrak{X}$.
\end{proposition}

\begin{proof}
The proof is completely analogous to \cite[Proposition 6]{BrunellaMiari}.    
\end{proof}

It can be defined the \emph{lower principal part} and \emph{non degenerated lower principal part}. Indeed, the main theorem of \cite{Berezovskaya,BrunellaMiari} says that the topological behavior of the singularity is determined by the lower principal part of an analytic vector field. The idea is to show that, given an analytic vector field with non degenerated lower principal part, the singularities on the exceptional divisor are determined by the terms of the lower principal part. In Theorem \ref{main-theorem}, we show that the singularities at the infinity are determined by the terms of the upper principal part.

\begin{definition}
A planar vector field $Y$ is \emph{free of characteristic orbits near the infinity} if, for its compactification $\overline{Y}$, there is no singularity in the boundary of $\mathbb{D}_{(\alpha,\beta)}$ having characteristic orbit intersecting the interior of $\mathbb{D}_{(\alpha,\beta)}$.     
\end{definition}

A precise definition of characteristic orbit can be found in \cite[Definition 1.19]{Dumortier77}.

\begin{definition}\label{def-eq-infinity}
The planar polynomial vector fields $X_{1}$ and $X_{2}$ are \emph{topologically equivalent near the infinity of the Poincaré--Lyapunov disk $\mathbb{D}_{(\alpha,\beta)}$} (or simply topologically equivalent near the infinity for short) if there is a homeomorphism $H:W_{1}\rightarrow W_{2}$ satisfying:
\begin{description}
    \item[(H1)] $W_{1}$ and $W_{2}$ are open sets containing the boundary $\mathcal{D}$ of the PL disk $\mathbb{D}_{(\alpha,\beta)}$, and $H(\mathcal{D}) = \mathcal{D}$.
    \item[(H2)] Denote the PL compactification of $X_{i}$ by $\overline{X}_{i}$.  Given $t > 0$ and $p\in W_{1}$, there is $t^{*} > 0$ such that
    $$H\Big{(}\varphi_{\overline{X}_{1}}(p,[0,t])\Big{)} = \varphi_{\overline{X}_{2}}(H(p),[0,t^{*}]),$$
    in which $\varphi_{\overline{X}_{i}}$ is the flow of $\overline{X}_{i}$, for $i = 1,2$. See Figure \ref{fig-eq-infinity}.
\end{description}
\end{definition}

\begin{figure}[h]\center{
\begin{overpic}[width=0.80\textwidth]{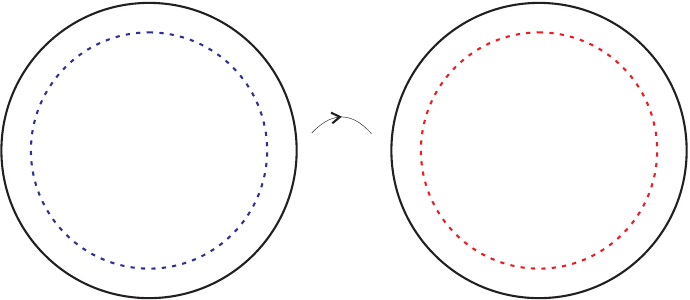}
\put(20,-4){$\mathcal{D}$}
\put(79,-4){$\mathcal{D}$}
\put(6,20){$W_{1}$}
\put(89,20){$W_{2}$}
\put(47,29){$H$}
\end{overpic}}
\caption{\footnotesize{Topological equivalence near the infinity. See Definition \ref{def-eq-infinity}.}}
\label{fig-eq-infinity}
\end{figure}
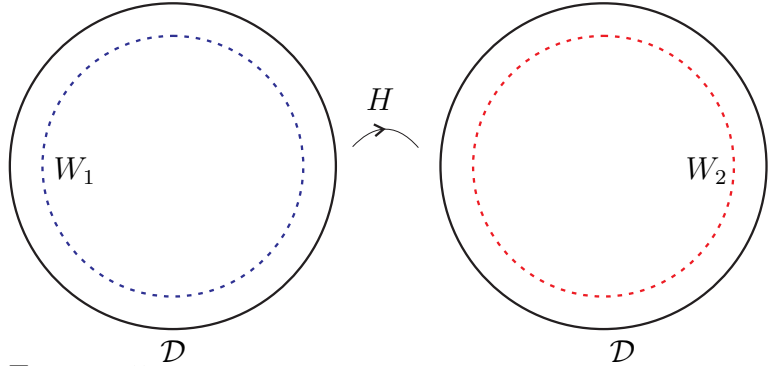

\begin{mtheorem}\label{main-theorem}
Let $X$ be a planar polynomial vector field with non degenerated upper principal part and that is not free of characteristic orbits near the infinity. Suppose also that  $X_{\Delta}^{U}$ does not have curve of singularities. Then $X$ and $X_{\Delta}^{U}$ are topologically equivalent near the infinity.
\end{mtheorem}

\begin{proof}
The proof of the Theorem will consider the case in which the vectors of the skeleton $\Sigma^{*}_{\mathcal{P}}$ do not have positive entries. In other words, for simplicity sake, we consider the case where the segments of the upper boundary $\mathcal{P}^{U}$ has negative slope. The calculations are completely analogous in the case which the upper boundary contains segments that has positive slope, or that are horizontal or vertical.

Let $\Sigma_{\mathcal{P}}$ be the simple fan associated to $\mathcal{P}^{U}$. Since we are considering the case where the vectors of $\Sigma^{*}_{\mathcal{P}}$ do not have positive entries, then $\xi_{1} = (-1,-\beta_{1})$ and $\xi_{s-1} = (-\alpha_{s-1},-1)$ with $\alpha_{s-1},\beta_{1} \geq 0$.

For each $j = 1,\dots, s-1$ define
$$M_{j} = \min_{(m,n)\in\mathcal{P}}\{\langle\xi_{j},(m,n)\rangle\}; \quad \ \Gamma_{j} = \{(m,n)\in\mathcal{P} \ ; \ \langle\xi_{j},(m,n)\rangle = M_{j} \}.$$

Observe that each $M_{j}$ is a negative integer, because the fan $\Sigma_{\mathcal{P}}$ is obtained from the collection of inward normal vectors $\Sigma^{*}_{\mathcal{P}}$. It will also be supposed that $M_{0} = M_{s} = 0$. Moreover,  $\Gamma_{j}$ is the set of points of $\mathcal{P}$ such that the minimum $M_{j}$ is reached. It must be clear that each $\Gamma_{j}$ is either a vertex or a segment of the upper boundary $\mathcal{P}^{U}$. The proof follows analysing the compactification in each set $U_{j}$. In what follows, a singularity is \emph{semi-hyperbolic} if the linearization (Jacobian matrix) of the vector field in such point has exactly one non zero eigenvalue.

\textbf{Case (1): Change of coordinates induced by $\xi_{0} = (0,1)$ and $\xi_{1} = (-1,-\beta_{1})$ with $\beta_{1} \geq 0$.} We obtain
$$x = \frac{1}{v}, \ y = \frac{u}{v^{\beta_{1}}}, \ \quad \ u = \frac{y}{x^{\beta_{1}}}, \ v = \frac{1}{x}.$$

From differential geometry, we have
$$\displaystyle\frac{\partial}{\partial x} = \displaystyle\frac{\partial u}{\partial  x}\displaystyle\frac{\partial}{\partial u} + \displaystyle\frac{\partial v}{\partial x}\displaystyle\frac{\partial}{\partial v}, \ \quad \ \displaystyle\frac{\partial}{\partial y} = \displaystyle\frac{\partial u}{\partial  y}\displaystyle\frac{\partial}{\partial u} + \displaystyle\frac{\partial v}{\partial y}\displaystyle\frac{\partial}{\partial v};$$
and therefore we obtain the vector field
$$
\widetilde{X}_{1}(u,v) = \sum_{(m,n)\in\mathcal{P}}\displaystyle\frac{u^{n}}{v^{m + \beta_{1}n}}\Big{(}(b_{m,n} - \beta_{1}a_{m,n})u\frac{\partial}{\partial u} - a_{m,n}v\frac{\partial}{\partial v}\Big{)}.$$

By multiplying $\widetilde{X}_{1}$ by $v^{|M_{1}|}$, the vector field obtained is
\begin{equation}\label{eq-main-theorem-compact-chart-1}
\overline{X}_{1}(u,v) = \sum_{(m,n)\in\mathcal{P}} u^{n}v^{|M_{1}| + \langle\xi_{1},(m,n)\rangle}\Big{(}(b_{m,n} - \beta_{1}a_{m,n})u\frac{\partial}{\partial u} - a_{m,n}v\frac{\partial}{\partial v}\Big{)}
\end{equation}

The line $\{v = 0\}$ plays the role of the infinity. There are two cases that must be considered.

\textbf{(1.1) The vectors $\xi_{0}$ and $\xi_{1}$ are normal to no segment of $\mathcal{P}^{U}$}. In this case. $\Gamma_{1}$ is a vertex of $\mathcal{P}^{U}$, that is, $M_{1}$ is reached in only one point of the support. If such a vertex is of the form $(\tilde{m},-1)$, then $b_{\tilde{m},-1} \neq 0$ and there is no singularities at the infinity $\{v = 0\}$ for $\overline{X}_{1}$.

On the other hand, if the vertex is of the form $(\tilde{m},0)$, then $a_{\tilde{m},0}^{2} + b_{\tilde{m},0}^{2} \neq 0$. In fact, the coefficient $a_{\tilde{m},0}$ is always non zero, otherwise the vector field $X_{\Delta}^{U}$ would have a curve of singularities arriving at the infinity, and it contradicts our hypothesis. If $(b_{\tilde{m},0} - \beta_{1}a_{\tilde{m},0}) \neq 0$, in the infinity $\{v = 0\}$ the vector field $\overline{X}_{1}$ has only one equilibrium point, which is positioned at the origin and in the worst scenario it is semi-hyperbolic because $a_{\tilde{m},0}\neq 0$, and the center manifold locally coincide with $\{v = 0\}$. Finally, if $(b_{\tilde{m},0} - \beta_{1}a_{\tilde{m},0}) = 0$ then the infinity itself is a curve of singularities, and all points are semi-hyperbolic. In brief, the topological behavior of the singularities at the infinity depend only on the terms associated to $\mathcal{P}^{U}$.

\textbf{(1.2) The vector $\xi_{1}$ is normal to a segment $\gamma_{l}\subset\mathcal{P}^{U}$}. In this case, $\Gamma_{1} = \gamma_{l}$, for some $l$. The singularities of $\overline{X}_{1}$ in the infinity $\{v = 0\}$ are given by the roots of the polynomial
\begin{equation*}
\overline{X}_{1}(u,0) = \sum_{(m,n)\in\gamma_{l}} u^{n}(b_{m,n} - \beta_{1}a_{m,n})u.
\end{equation*}

Let us study the linearization of $\overline{X}_{1}$ at a singularity $(\bar{u},0)$. Denote $\overline{X}_{1} = \overline{X}_{1,u}\frac{\partial}{\partial u} + \overline{X}_{1,v}\frac{\partial}{\partial v}$. Observe that $\frac{\partial \overline{X}_{1,v}}{\partial u}(\bar{u},0) = 0$, thus the Jacobian matrix of $\overline{X}_{1}$ at $(\bar{u},0)$ is given by
$$
J\overline{X}_{1}(\bar{u},0) = \left(
  \begin{array}{cc}
    \frac{\partial \overline{X}_{1,u}}{\partial u} & \frac{\partial \overline{X}_{1,u}}{\partial v} \\
    0 & \frac{\partial \overline{X}_{1,v}}{\partial v} \\
  \end{array}
\right).
$$

This means that the eigenvalues of $\overline{X}_{1}$ at $(\bar{u},0)$ depend only on $\frac{\partial \overline{X}_{1,u}}{\partial u}$ and $\frac{\partial \overline{X}_{1,v}}{\partial v}$. We claim that the eigenvalue $\frac{\partial \overline{X}_{1,v}}{\partial v}$ is always non zero, that is, in the worst scenario the singularity is semi-hyperbolic. Indeed, suppose that
$$
\overline{X}_{1}(\bar{u},0) =  0; \ \quad \ \frac{\partial \overline{X}_{1,v}}{\partial v}(\bar{u},0) = 0; \ \quad \ \text{with} \ \bar{u} \neq 0.$$

Rewriting the last system of equations, it follows that
\begin{align*}
&\left\{
  \begin{array}{ccl}
    \displaystyle\sum_{(m,n)\in\gamma_{l}}\bar{u}^{n+1}(b_{m,n}-\beta_{1}a_{m,n}) & = & 0; \\
    -\displaystyle\sum_{(m,n)\in\gamma_{l}}\bar{u}^{n}a_{m,n} & = & 0; 
  \end{array}
\right.\Rightarrow \\
& \Rightarrow
\left(
  \begin{array}{cc}
    -\beta_{1} & 1 \\
    -1 & 0 \\
  \end{array}
\right) 
\left(
  \begin{array}{c}
    \displaystyle\sum_{(m,n)\in\gamma_{l}}\bar{u}^{n}a_{m,n} \\
    \displaystyle\sum_{(m,n)\in\gamma_{l}}\bar{u}^{n+1}b_{m,n} \\
  \end{array}
\right) = \left(
  \begin{array}{c}
    0 \\
    0 \\
  \end{array}
\right) \Rightarrow  \left\{
  \begin{array}{rc}
    P_{\gamma_{l}}(1,\bar{u}) & = 0; \\
    Q_{\gamma_{l}}(1,\bar{u}) & = 0. 
  \end{array}
\right. 
\end{align*}

The last system of equations would imply that the vector field $X_{\gamma_{j}}$ has singularities in $(\mathbb{R}^{*})^{2}$, which is a contradiction. Thus the eigenvalue $\frac{\partial \overline{X}_{1,v}}{\partial v}$ is always non zero for $(\bar{u},0)$ with $\bar{u}\neq 0$. On the other hand, if the singularity is positioned at the origin (that is, $\bar{u} = 0$), it follows from equation \eqref{eq-main-theorem-compact-chart-1} that in the worst scenario such a singularity is semi-hyperbolic. In both cases, the center manifold locally coincides with $\{v = 0\}$ and therefore the behavior of the singularities at the infinity depend only on terms of the upper boundary $\mathcal{P}^{U}$.

\textbf{Case (2): Change of coordinates induced by the vectors $\xi_{s-1} = (-\alpha_{s-1},-1)$ and $\xi_{s} = (1,0)$ with $\alpha_{s-1} \geq 0$.} We obtain
$$x = \frac{v}{u^{\alpha_{s-1}}}, \ y = \frac{1}{u}, \ \quad \ u = \frac{1}{y}, \ v = \frac{x}{y^{\alpha_{s-1}}}.$$

After change of coordinates and then multiplying the vector field by $u^{|M_{s-1}|}$ in order to extend it to the infinity, one obtains the vector field
\begin{equation}
\overline{X}_{s}(u,v) = \sum_{(m,n)\in\mathcal{P}} u^{|M_{s-1}| + \langle\xi_{s-1},(m,n)\rangle}v^{m}\Big{(}-b_{m,n}u\frac{\partial}{\partial u} + (a_{m,n}-\alpha_{s-1}b_{m,n})v\frac{\partial}{\partial v}\Big{)}.
\end{equation}

The analysis of the singularities of the vector field $\overline{X}_{s}$ at the infinity $\{u = 0\}$ is completely analogous to Case (1).

\textbf{Case (3): Change of coordinates induced by the vectors $\xi_{j-1} = (-\alpha_{j-1},-\beta_{j-1})$ and $\xi_{j} = (-\alpha_{j},-\beta_{j})$, with $\alpha_{j-1},\beta_{j-1},\alpha_{j},\beta_{j}$ being non negative.} One obtains the following change of coordinates
$$x = \frac{1}{u^{\alpha_{j-1}}v^{\alpha_{j}}}, \ y = \frac{1}{u^{\beta_{j-1}}v^{\beta_{j}}}, \ \quad \ u = \frac{y^{\alpha_{j}}}{x^{\beta_{j}}}, \ v = \frac{x^{\beta_{j-1}}}{y^{\alpha_{j-1}}}.$$

Once again, after change of coordinates and then multiplying the vector field by $u^{|M_{s-1}|}$ in order to extend it to the infinity, one obtains the vector field
\begin{equation}
\overline{X}_{j}(u,v) = \sum_{(m,n)\in\mathcal{P}} u^{|M_{j-1}| + \langle\xi_{j-1},(m,n)\rangle}v^{|M_{j}| + \langle\xi_{j},(m,n)\rangle}\Big{(}A_{m,n}u\frac{\partial}{\partial u} + B_{m,n}v\frac{\partial}{\partial v}\Big{)}.
\end{equation}
in which $A_{m,n} = (-\beta_{j}a_{m,n} + \alpha_{j}b_{m,n})$ and $B_{m,n} = (\beta_{j-1}a_{m,n}-\alpha_{j-1}b_{m,n})$. There are three subcases that must be considered.

\textbf{(3.1) The vectors $\xi_{j-1}$ and $\xi_{j}$ are normal to no segment of $\mathcal{P}^{U}$.} In this case, the set $\Gamma_{j-1} = \Gamma_{j}$ is a vertex $p_{k} = (m_{k},n_{k})\in\mathcal{P}^{U}$. Moreover, the linearization of $\overline{X}_{j}$ is given by
$$
J\overline{X}_{j}(u,v) = \left(
  \begin{array}{cc}
    -\beta_{j} & \alpha_{j} \\
    \beta_{j-1} & -\alpha_{j-1} \\
  \end{array}
\right)\left(
  \begin{array}{c}
    a_{m_{k},n_{k}}u  \\
    b_{m_{k},n_{k}}v  \\
  \end{array}
\right). 
$$

If both $a_{m_{k},n_{k}}, b_{m_{k},n_{k}}$ are non zero, then the origin is the only singularity at the infinity, and it is hyperbolic because $\tiny\det\left(
  \begin{array}{cc}
    -\beta_{j} & \alpha_{j} \\
    \beta_{j-1} & -\alpha_{j-1} \\
  \end{array}
\right) = 1$. On the other hand, if $a_{m_{k},n_{k}} = 0$ or $b_{m_{k},n_{k}} = 0$, then the infinity contains a curve of singularities (all of them semi-hyperbolic). Therefore, the singularities depend only on terms of $\mathcal{P}^{U}$

\textbf{(3.2) The vector $\xi_{j}$ is normal to a segment $\gamma_{l}\subset\mathcal{P}^{U}$.} The analysis of the origin in completely analogous to the sub case (3.1). The other singularities are analyzed using the same strategy of the sub case (1.2). Observe that in this case the singularities at the infinity are positioned on the set $\{v = 0\}$. Once again, all singularities are either hyperbolic or semi hyperbolic, and their topological behavior depends on the terms of $\mathcal{P}^{U}$.

\textbf{(3.3) The vector $\xi_{j-1}$ is normal to a segment $\gamma_{l}\subset\mathcal{P}^{U}$.} The analysis is completely analogous to the sub case (3.2). The only difference is that all the singularities at the infinity are positioned on the set $\{u = 0\}$. All singularities are either hyperbolic or semi hyperbolic, and their topological behavior depends on the monomials associated to the points of $\mathcal{P}^{U}$.

The proof so far guarantees that the vector field $X$ and its upper principal part $X_{\Delta}^{U}$ are topologically equivalent near each singularity at the infinity. The conclusion follows by proving that $X$ and $X_{\Delta}^{U}$ are topologically equivalent near the \emph{whole} boundary of the Poincaré--Lyapunov disk. Since all the singularities are either hyperbolic or semi hyperbolic and the infinity is the finite union of smooth arcs which is homeomorphic to $\mathbb{S}^{1}$, following the ideas and arguments of \cite[Theorem B]{Dumortier77} it can be constructed a homeomorphism that gives the topological equivalence between $X$ and $X_{\Delta}^{U}$ in a neighbourhood of the infinity (see also \cite[Theorem A]{PerezSilva2} for an analogous construction in the context of constrained differential systems).

In what follows we briefly describe the construction of the homeomorphism, and we refer to \cite[Theorem B]{Dumortier77} for details. The infinity is decomposed in a finite number of sectors, and in each sector there is either a (hyperbolic or semi-hyperbolic) singularity, or the infinity itself is a curve of singularities. From the calculations in the first part of the proof, we already know that $X$ and $X_{\Delta}^{U}$ are topologically equivalent in each sector separately. Applying the \emph{pasting lemma} (see \cite[Lemma 5.7]{Dumortier77}) one can ``glue'' two homeomorphisms defined in two adjacent sectors. Applying this reasoning to all sectors, we now have the topological equivalence defined in a neighbourhood of the \emph{whole} infinity. Recall that the infinity is homeomorphic to $\mathbb{S}^{1}$, so in the end of the construction one must glue the last and the first sector. A crucial hypothesis to end this construction is that at least one singularity at the infinity has characteristic orbit. \end{proof}

\subsection{Periodic orbits near the infinity}\label{subsec-periodic-orbits}

Chicone and Sotomayor \cite{ChiconeSoto} addressed the problem of completeness of integral curves of planar polynomial vector fields, and in particular in the case which the infinity itself is a periodic orbit. Some of their results related the multiplicity of the periodic orbit at the infinity with the notion of completeness of integral curves. It was proved that, if the infinity is a periodic orbit of infinite multiplicity, then there is an open neighbourhood of the infinity filled with periodic orbits.

In this subsection it will be considered the case which $X$ is free of characteristic orbits near the infinity and the upper boundary is just a segment, that is, $\mathcal{P}^{U} = \gamma$. We emphasize that in the statement of Theorem \ref{main-theorem-particular-case} it is necessary to strengthen the non degeneracy hypothesis. In Theorem \ref{main-theorem} we supposed that each $X_{\gamma_{j}}$ does not have singularities outside the coordinate axes, whereas in Theorem \ref{main-theorem-particular-case} it will be supposed that the only singularity of $X^{U}_{\Delta}$ is the origin. In this proof, it is not necessary to carry out the construction of the toric variety associated to the Newton polytope.

\begin{definition}
A planar vector field $Y$ is \emph{free of periodic orbits near the infinity} if there is a neighbourhood $W$ of the boundary of $\mathbb{D}_{(\alpha,\beta)}$ such that the compactification $\overline{Y}$ does not have periodic orbits in $W$.
\end{definition}

\begin{mtheorem}\label{main-theorem-particular-case}
Let $X$ be a planar polynomial vector field free of characteristic orbits near the infinity and suppose that the upper boundary is just a segment, that is, $\mathcal{P}^{U} = \gamma$. Suppose also that the origin is the only singularity of $X_{\Delta}^{U}$. If $X_{\Delta}^{U}$ is free of periodic orbits near the infinity, then $X$ is also free of periodic orbits near the infinity.   
\end{mtheorem}

\begin{proof}
Suppose that the inward normal vector of $\gamma$ is $(-\alpha,-\beta)$. Using the change of coordinates \eqref{eq-cs-sn}, one obtain the following vector field
\begin{equation}\label{eq-compact-polar}
\begin{split}
\overline{X}(\theta, & r) = \displaystyle\sum_{d = -1}^{\delta - 1}r^{\delta - d - 1}\Big{(}\operatorname{Cs}\theta Q_{d}(\operatorname{Cs}\theta,\operatorname{Sn}\theta) - \displaystyle\frac{\beta}{\alpha}\operatorname{Sn}\theta P_{d}(\operatorname{Cs}\theta,\operatorname{Sn}\theta)\Big{)} \displaystyle\frac{\partial}{\partial \theta} \\
& -\displaystyle\sum_{d = -1}^{\delta - 1}\displaystyle\frac{r^{\delta - d}}{\alpha}\Big{(}\operatorname{Cs}^{2\beta - 1}
\theta P_{d}(\operatorname{Cs}\theta,\operatorname{Sn}\theta) + \operatorname{Sn}^{2\alpha - 1}
\theta Q_{d}(\operatorname{Cs}\theta,\operatorname{Sn}\theta)\Big{)} \displaystyle\frac{\partial}{\partial r}.
\end{split}
\end{equation}

In the notation above, $P_{d}$ is a quasi homogeneous polynomial of type $(\alpha,\beta)$ and degree $d + \alpha$, whereas $Q_{d}$ is a quasi homogeneous polynomial of type $(\alpha,\beta)$ and degree $d + \beta$. The highest quasi homogeneous components are $P_{\delta-1}$ and $Q_{\delta - 1}$.

For simplicity sake, write the Equation \eqref{eq-compact-polar} as
$$\overline{X}(\theta, r) = \overline{X}_{\theta}(\theta, r)\displaystyle\frac{\partial}{\partial \theta} + \overline{X}_{r}(\theta, r)\displaystyle\frac{\partial}{\partial \theta}.$$

The Jacobian matrix of \eqref{eq-compact-polar} in a singularity in the infinity $\{r = 0\}$ is given by
$$
J(\theta, 0) = \left(
  \begin{array}{cc}
    \displaystyle\frac{\partial \overline{X}_{\theta}}{\partial \theta}(\theta, 0) & \displaystyle\frac{\partial \overline{X}_{\theta}}{\partial r}(\theta, 0) \\
    0 & \displaystyle\frac{\partial \overline{X}_{r}}{\partial r}(\theta, 0) \\
  \end{array}
\right);
$$
and therefore the eigenvalues depend on $\frac{\partial \overline{X}_{\theta}}{\partial \theta}(\theta, 0)$ and $\frac{\partial \overline{X}_{r}}{\partial r}(\theta, 0)$. We affirm that if $(\bar{\theta},0)$ is a singularity, then $\frac{\partial \overline{X}_{r}}{\partial r}(\bar{\theta},0) \neq 0$, that is, in the worst scenario the singularity is semi hyperbolic. Indeed, supposing that
$\overline{X}_{\theta}(\bar{\theta}, 0) = 0$ and $ \frac{\partial \overline{X}_{r}}{\partial r}(\bar{\theta},0) = 0$ one would have
\begin{equation*}
\begin{split}
& \left\{
  \begin{array}{rcl}
    \overline{X}_{\theta}(\bar{\theta}, 0) & = & 0; \\
    \frac{\partial \overline{X}_{r}}{\partial r}(\bar{\theta},0) & = & 0;
  \end{array}
\right. \Rightarrow \\
& \Rightarrow
\left(
  \begin{array}{cc}
    -\frac{\beta}{\alpha}\operatorname{Sn}\bar{\theta} & \operatorname{Cs}\bar{\theta} \\
    \operatorname{Cs}^{2\beta - 1}\bar{\theta} & \operatorname{Sn}^{2\alpha - 1}\bar{\theta} \\
  \end{array}
\right)\left(
         \begin{array}{c}
           P_{\delta - 1}(\operatorname{Cs}\bar{\theta},\operatorname{Sn}\bar{\theta}) \\
           Q_{\delta - 1}(\operatorname{Cs}\bar{\theta},\operatorname{Sn}\bar{\theta}) \\
         \end{array}
       \right)
=
\left(
         \begin{array}{c}
           0 \\
           0 \\
         \end{array}
       \right) \Rightarrow \\
       & \Rightarrow\left\{
  \begin{array}{rcl}
    P_{\delta - 1}(\operatorname{Cs}\bar{\theta},\operatorname{Sn}\bar{\theta}) & = & 0; \\
    Q_{\delta - 1}(\operatorname{Cs}\bar{\theta},\operatorname{Sn}\bar{\theta}) & = & 0;
  \end{array}
\right.
\end{split}
\end{equation*}
which is a contradiction because by hypothesis the origin is the only singularity of $X^{U}_{\Delta}$ and $(\operatorname{Cs}\theta,\operatorname{Sn}\theta) \neq (0,0)$ for all $\theta$. Therefore, the singularities at infinity of $X$ are either hyperbolic or semi hyperbolic, and they depend only on terms of $X^{U}_{\Delta}$. 

In order to study the existence of periodic orbits in a neighbourhood of the infinity $\{r = 0\}$, we will compute the Poincaré map for $\overline{X}$ and $\overline{X}^{U}_{\Delta}$. Consider the ODEs
\begin{equation*}
\displaystyle\frac{dr}{d\theta} = \displaystyle\frac{\overline{X}_{r}(\theta,r)}{\overline{X}_{\theta}(\theta,r)} = G(\theta)r + \mathcal{O}(r^{2}), \ \quad \ \displaystyle\frac{dr}{d\theta} = \displaystyle\frac{\overline{X}^{U}_{\Delta,r}(\theta,r)}{\overline{X}^{U}_{\Delta,\theta}(\theta,r)} = G^{U}_{\Delta}(\theta)r,
\end{equation*}
whose solutions are denoted by $r(\theta)$ and $r^{U}_{\Delta}(\theta)$, respectivelly. Recall that the period $T$ of the functions $\operatorname{Cs}{\theta}$ and $\operatorname{Sn}\theta$ is given by \eqref{eq-period-lyapunov}. If $X^{U}_{\Delta}$ is free of periodic orbits near the infinity, then there is a neighbourhood $W$ of $\{r = 0\}$ such that
\begin{equation*}
r^{U}_{\Delta}(T) - r^{U}_{\Delta}(0)  = e^{\int_{0}^{T}G^{U}_{\Delta}(\tau)d\tau} - 1 \neq 0.    
\end{equation*}
and therefore $r^{U}_{\Delta}(T) - r^{U}_{\Delta}(0)$ and $r(T) - r(0)$ have the same sign in such neighbourhood, which implies that $X$ is also free of periodic orbits near the infinity. \end{proof}

\section{The effect of PLC process in the Newton polytope}\label{sec-plc-newton-polygon}

The goal of this section is to study the effect of the PLC process on the Newton polytope of a planar polynomial vector field. The effect of the \emph{Poincaré-compactification} in the \emph{Newton polyhedra} of a \emph{$n$-dimensional} polynomial vector field was already discussed in \cite[Proposition 4]{Kappos} (see also \cite{Berezovskaya2} for discussion in the planar case). Pelletier \cite{Pelletier} has investigated  the effect of quasi homogeneous blow-ups in the Newton polytope.

Recall that the Newton polygon depends on the coordinate system adopted. Given a polynomial vector field $X$, we say that $\mathcal{P}_{X}$ is \emph{favorable} if $\mathcal{P}^{U}_{X}$ has at least one segment with negative slope. See Figure \ref{fig-def-favorable}.

\begin{figure}[h]\center{
\begin{overpic}[width=0.95\textwidth]{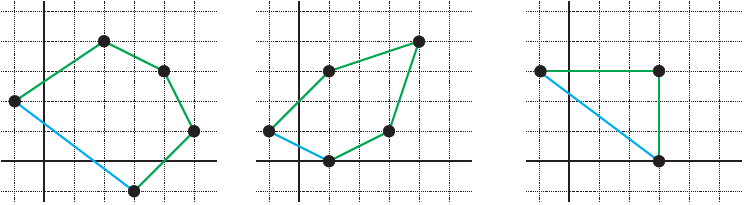}
\end{overpic}}
\caption{\footnotesize{Favorable Newton polytope (left). The other two polytopes are non favorable. The upper and lower boundaries $\mathcal{P}^{U}$ and $\mathcal{P}^{L}$ are highlighted in green and blue, respectively.}}
\label{fig-def-favorable}
\end{figure}

Let $p_{0}\in\mathcal{P}^{L}$ be the first point of the support $\mathcal{S}$ according to the Lexicographical order, and  $p_{h} = (m_{h}, n_{h})\in\mathcal{P}^{U}$ be the last point of $\mathcal{S}$ according to the Lexicographical order satisfying the condition $n_{h} \geq n$, for all $p = (m, n)\in\mathcal{P}$. The vertices $p_{0}$ and $p_{h}$ are called \emph{lower} and \emph{upper main vertices}, respectively. Analogously, the segments $\gamma_{1}$ and $\gamma_{h}$ (recall from subsection \ref{subsec-newton-polytope} the way we enumerated the segments of $\partial\mathcal{P}$) are called \emph{lower} and \emph{upper main segments}, respectively. See Figure \ref{fig-lower-upper-segments}.

\begin{figure}[h]\center{
\begin{overpic}[width=0.50\textwidth]{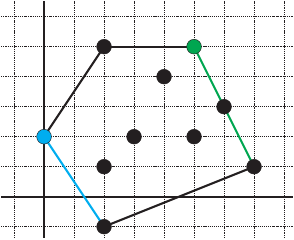}
\put(5,30){$p_{0}$}
\put(23,25){$\gamma_{1}$}
\put(70,70){$p_{h}$}
\put(65,50){$\gamma_{h}$}
\end{overpic}}
\caption{\footnotesize{The lower and upper main vertices $p_{0}$ and $p_{h}$ are highlighted in blue and green, respectively. Analogously, the lower and upper main segments $\gamma_{1}$ and $\gamma_{h}$ are highlighted in blue and green, respectively.}}
\label{fig-lower-upper-segments}
\end{figure}
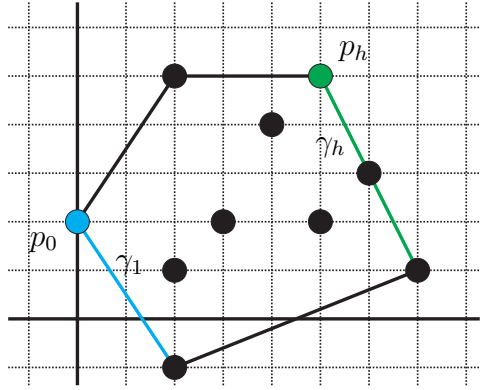

Recall that the segments of $\partial\mathcal{P}$ were enumerated in the counter clockwise sense. From now on, we will work with the segments $\gamma^{U}_{j}\subset\mathcal{P}^{U}$ of the upper boundary, and these segments will also be enumerated in the counter clockwise sense: $\gamma^{U}_{1},\dots,\gamma^{U}_{l-1},\gamma_{h},\gamma^{U}_{l+1},\dots,\gamma^{U}_{k}$.

\begin{definition}
Let $\mathcal{P}$ be a favorable Newton polytope. The \emph{right upper part of $\mathcal{P}^{U}$} is the union of the segments $\gamma^{U}_{1},\dots,\gamma^{U}_{l-1}\subset\mathcal{P}^{U}$, and the \emph{left upper part of $\mathcal{P}^{U}$} is the union of the segments $\gamma^{U}_{l+1},\dots,\gamma^{U}_{k}\subset\mathcal{P}^{U}$.
\end{definition}

In the resolution of isolated singularities of planar analytic vector fields using weighted blow-ups \cite{Pelletier}, we chose as weight vector $\omega$ the inward normal vector of the lower main segment $\gamma_{1}$. In what follows is indicated how to chose a suitable weight vector for the PLC process in the case which $\mathcal{P}$ is favorable. If $\mathcal{P}$ is non favorable, one can apply the classical Poincaré compactification in order to study the dynamics near the infinity.

\subsection{Effect of the Poincaré--Lyapunov compactification on a favorable Newton polytope}

Denote the line that contains the upper main segment $\gamma_{h}$ by $r_{h}$. The suitable weight vector for the Poincaré--Lyapunov compactification is given by $r_{h}$. In other words, if the line that contains $\gamma_{h}$ is given by $r_{h} = \{\alpha m + \beta n = \delta\}$, then the weight vector is $\omega = (\alpha,\beta)$. Note that $\delta$ is the highest level of the $(\alpha,\beta)$-decomposition of $X$ (recall Equation \eqref{eq-quasi-homogeneous-vf}). See Figure \ref{fig-segments}.

\begin{figure}[h]\center{
\begin{overpic}[width=0.6\textwidth]{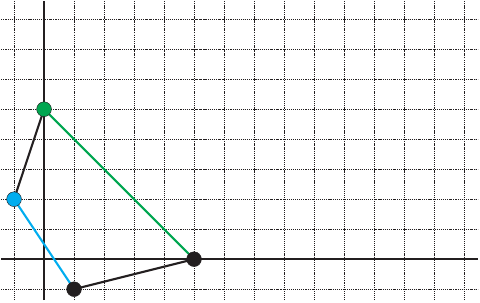}
\put(-5,21){$p_{0}$}
\put(15,40){$p_{h}$}
\put(3,5){$\gamma_{1}$}
\put(25,30){$\gamma_{h}$}
\end{overpic}}
\caption{\footnotesize{In the resolution of singularities, the weight vector is normal to the lower main segment $\gamma_{1}$ (which is higlighted in blue). On the other hand, in the Poincaré--Lyapunov compactification, the weight vector is normal to the upper main segment $\gamma_{h}$ (which is higlighted in green).}}
\label{fig-segments}
\end{figure}
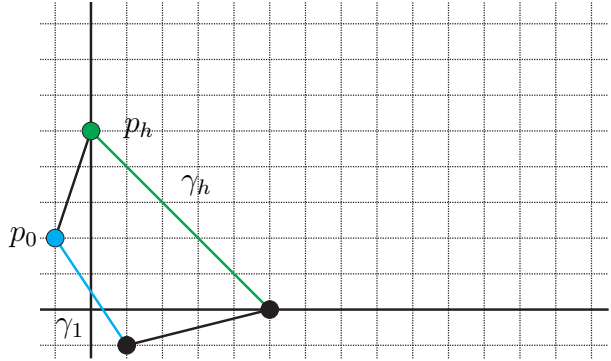

Our study starts by analyzing the compactification at the positive $y$-direction given by \eqref{eq-plc-y-positive-vf}. More precisely, Proposition \ref{prop-plc-y-dir} describes the effect of the compactification in the upper boundary $\mathcal{P}^{U}$. The analysis for the negative $y$ direction is completely analogous. In what follows, a point $q$ is an \emph{elementary point} of a given vector field if $q$ is a regular point, a hyperbolic or semi-hyperbolic singularity of the vector field.

\begin{proposition}\label{prop-plc-y-dir}
Let $X$ be a polynomial vector field given by \eqref{eq-quasi-homogeneous-vf}. After a Poincaré-Lyapunov compactification in the positive $y$-direction, the following hold:
\begin{description}
    \item[(a)] The lower boundary of $\mathcal{P}_{\overline{X}^{+}_{y}}$  is precisely the image of the left upper part of $\mathcal{P}^{U}_{X}$.
    \item[(b)] If $p_{h}\in\mathcal{P}_{X}$ is contained in $\{-1\}\times\mathbb{N}$ or $\{0\}\times\mathbb{N}$, then origin is an elementary point for $\overline{X}^{+}_{y}$.
\end{description}

\end{proposition}
\begin{proof}
 Let $r_{h} = \{\alpha m + \beta n = \delta\}$ be the line that contains the upper main segment $\gamma_{h}$. Observe that $X^{(\alpha,\beta)}_{\delta}$ is the highest level of the $(\alpha,\beta)$ decomposition \eqref{eq-quasi-homogeneous-vf}. After a compactification in the positive $y$-direction with weights $(\alpha,\beta)$, we obtain the vector field \eqref{eq-plc-y-positive-vf}.

The transformation \eqref{eq-plc-y-positive} does not modify powers of $x$, but it modifies powers of $y$. Moreover, before multiplying the vector field by $v^{\delta}$, every line of the form $\{\alpha m + \beta n = d\}$ is sent in a horizontal line of the form $\{(m,-d); m\in\mathbb{R}\}$. In particular, the line $r_{h}$ is sent to the line $\{(m,-\delta); m\in\mathbb{R}\}$, which is the first line of the new plane of powers (according to the lexicographical order). By multiplying the vector field $v^{\delta}$, the line $\{(m,-\delta); m\in\mathbb{R}\}$ is sent to the line $\{(m,0); m\in\mathbb{R}\}$, that is, $\gamma_{h}$ has been sent to a horizontal segment contained in the horizontal axis. See Figure \ref{fig-plc-y-dir}.

\begin{figure}[h]\center{
\begin{overpic}[width=0.9\textwidth]{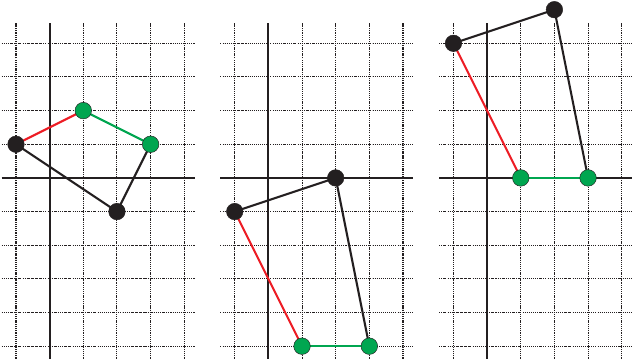}
\end{overpic}}
\caption{\footnotesize{The effect of the Poincaré--Lyapunov compactification in the Newton polytope. The leftmost polytope is the initial polytope, and the rightmost figure is the polytope after PLC in the positive $y$ direction. The upper main segment $\gamma_{h}$ is highlighted in green. See Proposition \ref{prop-plc-y-dir}.}}
\label{fig-plc-y-dir}
\end{figure}

It is important to stress that the compactification in the $y$-positive direction does not change powers of $x$, whereas it changes powers of $y$. In other words, if $(m,n)\in\mathcal{S}_{X}$ and $(m,n)$ is contained in the line $\{\alpha m + \beta n = d\}$, then after compactification in the positive $y$ direction it is mapped into the point $(m,\delta - d)$. This implies that the segments $\gamma^{U}_{l+1},\dots,\gamma^{U}_{k}$ are mapped into the segments of the lower boundary of $\mathcal{P}_{\overline{X}^{+}_{y}}$. In particular, for $j = l+1\dots,k$, if $\gamma^{U}_{j}$ is contained in a line of the form $\{\alpha' m + \beta'n = d'\}$ with $\gcd(\alpha',\beta') = 1$, then the slope of the segment $\gamma^{U}_{j}$ is $\theta' = -\frac{\alpha'}{\beta'}$ and after PLC in the $y$ direction the segment will have slope $\theta = \frac{\beta\alpha' - \beta'\alpha}{\beta'}$. This proves item (a).

For item (b), if $p_{h}\in\mathcal{P}_{X}$ is contained in $\{-1\}\times\mathbb{N}$ or $\{0\}\times\mathbb{N}$, after PLC $p_{h}$ is sent in a point of the form $(-1,0)$ or $(0,0)$, which implies that the origin is an elementary point. 
 \end{proof}




The next proposition aims to describe of the effect of the Poincaré--Lyapunov compactification in the positive $x$-direction on $\mathcal{P}^{U}$. More precisely, Proposition \ref{prop-plc-x-dir} states what happens to $\mathcal{P}^{U}$ after a transformation of the form \eqref{eq-plc-x-positive}.

\begin{proposition}\label{prop-plc-x-dir}
Let $X$ be a polynomial vector field given by \eqref{eq-quasi-homogeneous-vf}. After Poincaré--Lyapunov compactification in the positive $x$-direction, the lower boundary of $\mathcal{P}_{\overline{X}^{+}_{x}}$  is the image of the right upper part of $\mathcal{P}^{U}_{X}$.
\end{proposition}

\begin{proof}
Firstly, let us notice that the formula \eqref{eq-plc-x-positive} of the PLC in positive $x$-direction can be seen as the composition of the following maps: \begin{multicols}{2}\noindent
\begin{equation}\label{eq-prop-plc-x-positive-1}
x = \tilde{u}^{-\alpha}, \ y = \tilde{v}\tilde{u}^{-\beta}    
\end{equation}
\begin{equation}\label{eq-prop-plc-x-positive-2}
 \tilde{u} = v, \ \tilde{v} = u.   
\end{equation} \end{multicols}

The proof of this proposition consists in study the effect of the transformations \eqref{eq-prop-plc-x-positive-1} and \eqref{eq-prop-plc-x-positive-2} in the Newton polytope. Let $r_{h} = \{\alpha m + \beta n = \delta\}$ be the line that contains the upper main segment $\gamma_{h}$. Observe that $X^{(\alpha,\beta)}_{\delta}$ is the highest level of the $(\alpha,\beta)$ decomposition \eqref{eq-quasi-homogeneous-vf}. After a transformation of the form \eqref{eq-prop-plc-x-positive-1}, lines of the form $\{\alpha m + \beta n = d\}$ are sent to vertical lines of the form $\{m = -d\}$. This implies that the line $r_{h}$ is sent to the vertical line $\{m = -\delta\}$, and therefore $\gamma_{h}$ is mapped into a vertical segment. After an suitable multiplication of a power of $\tilde{u}$, the segment is translated to the vertical axis.

The transformation \eqref{eq-prop-plc-x-positive-1} does not change powers of $y$, whereas it changes powers of $x$. In other words, if $(m,n)\in\mathcal{S}_{X}$ is contained in the line $\{\alpha m + \beta n = d\}$, then after the transformation \eqref{eq-prop-plc-x-positive-1} it is mapped into the point $(\delta - d,n)$. Thus the segments $\gamma^{U}_{1},\dots,\gamma^{U}_{l-1}$ are mapped into the segments of the lower boundary of $\mathcal{P}_{\overline{X}^{+}_{x}}$. In particular, for $j = 1\dots,l-1$, if $\gamma^{U}_{j}$ is contained in a line of the form $\{\alpha' m + \beta'n = d'\}$ with $\gcd(\alpha',\beta') = 1$, then after PLC in the $y$ direction the segment will have slope $\theta = \frac{\alpha\beta' - \beta\alpha'}{\alpha'}$.

Observe that \eqref{eq-prop-plc-x-positive-2} is a linear transformation, whose effect is to reflect the points of the Newton polygon over the line $\{m = n\}$ (see Figure \ref{fig-plc-x-dir}). This concludes the proof. \end{proof}

\begin{figure}[h]\center{
\begin{overpic}[width=0.9\textwidth]{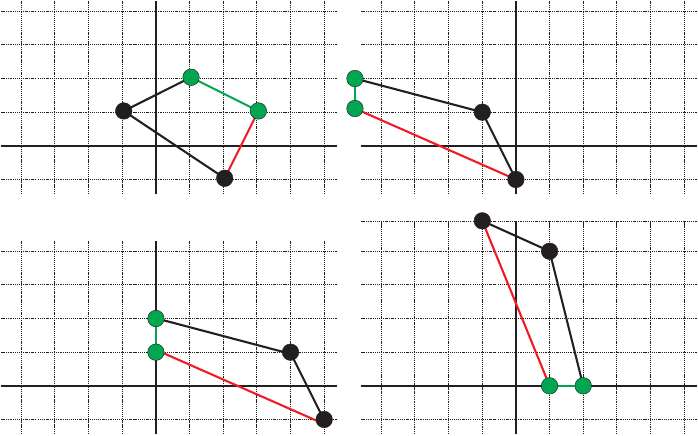}
\end{overpic}}
\caption{\footnotesize{The effect of the Poincaré--Lyapunov compactification in the positive $x$ direction. The leftmost figure in the top is the initial polytope, and the leftmost figure in the bottom figure is the polytope after the transformation \eqref{eq-prop-plc-x-positive-1}. Afterwards, applying a transformation of the form \eqref{eq-prop-plc-x-positive-2}, one obtains the rightmost polytope in the bottom. The upper main segment $\gamma_{h}$ is highlighted in green. See Proposition \ref{prop-plc-x-dir}.}}
\label{fig-plc-x-dir}
\end{figure}

\subsection{The non degeneracy condition and the PLC process}\label{subsec-influence-non-degeneracy}

Recall that the upper principal part of a planar vector field $X$ is non degenerated if the restriction $X_{\gamma_{j}}$ of $X$ does not have singularities outside the coordinate axes, in which $\gamma_{j}\subset\mathcal{P}^{U}$. The following proposition assures that this non degeneracy condition implies that the singularities at the infinity outside the origin of the $x$ and $y$ directions are either hyperbolic or semi-hyperbolic.

\begin{proposition}
Let $X$ be a planar polynomial vector field such that the upper principal part $X^{U}_{\Delta}$ is non degenerated. After a suitable PLC with weight $\omega = (\alpha, \beta)$, all the singularities in the infinity outside the origin of the (positive and negative) $x$ and $y$ directions are either hyperbolic or semi-hyperbolic.   
\end{proposition}
\begin{proof}
Consider the planar polynomial vector field
$$X(x,y) = P(x,y)\displaystyle\frac{\partial}{\partial x} + Q(x,y)\displaystyle\frac{\partial}{\partial y}.$$

Applying the PLC technique in the positive $x$ direction, one obtains
\begin{equation}\label{eq-prop-non-deg-condition}
\overline{X}^{+}_{x}(u,v) = \displaystyle\sum_{d = -1}^{\delta - 1}v^{\delta - d - 1}\Bigg{(}\Big{(}Q_{d}(1,u) - \frac{\beta}{\alpha}uP_{d}(1,u)\Big{)}\displaystyle\frac{\partial}{\partial u} -\displaystyle\frac{1}{\alpha}vP_{d}(1,u)\displaystyle\frac{\partial}{\partial v} \Bigg{)}.    
\end{equation}

Once again $P_{d}$ is a quasi homogeneous polynomial of type $(\alpha,\beta)$ and degree $d + \alpha$, whereas $Q_{d}$ is a quasi homogeneous polynomial of type $(\alpha,\beta)$ and degree $d + \beta$. Then the singularities at infinity $\{v = 0\}$ can be analyzed by the Jacobian matrix
$$
J(u,0) = \left(
  \begin{array}{cc}
    \frac{\partial Q_{\delta}}{\partial u}(1,u) - \frac{\beta}{\alpha}\Big{(}P_{\delta}(1,u) + u\frac{\partial P_{\delta}}{\partial u}(1,u)\Big{)} & * \\
    0 & -\frac{1}{\alpha}P_{\delta}(1,u) \\
  \end{array}
\right). 
$$

Suppose that $(\bar{u},0)$ is a singular point with $\bar{u}\neq 0$. If $P_{\delta}(1,\bar{u}) \neq 0$, it follows from $J(\bar{u},0)$ that in the worst scenario the singularity is semi hyperbolic. On the other hand, if $P_{\delta}(1,\bar{u}) = 0$, then $Q_{\delta}(1,\bar{u}) \neq 0$ because the upper principal part is non degenerated. But from Equation \eqref{eq-prop-non-deg-condition} this would imply that $(\bar{u},0)$ is not a singularity, which is a contradiction. Therefore the singularities at the infinity of the positive $x$ direction outside the origin are either hyperbolic or semi hyperbolic. With completely analogous arguments one can show that the singularities at the infinity in the (positive and negative) $y$ direction outside the origin are either hyperbolic or semi hyperbolic too.
\end{proof}

\subsection{Favorable change of coordinates}

It is important to remark that we can always assume that a polytope $\mathcal{P}$ is favorable, due to the change of coordinates described in Lemma \ref{lemma-preparacao-poligono}. However, after such change of coordinates, $X^{U}_{\Delta}$ may be degenerated in the sense of Definition \ref{def-upper-degenerated}. These remarks will be discussed in what follows.

\begin{lemma}\label{lemma-preparacao-poligono}
Let $X$ be a polynomial vector field given by \eqref{eq-quasi-homogeneous-vf} and denote its support and Newton polytope by $\mathcal{S}_{X}$ and $\mathcal{P}_{X}$, respectively. After a linear change of coordinates of the form
\begin{equation}\label{eq-change-coord-1}
\widetilde{x} = x - \lambda y, \ \quad \ \widetilde{y} = y, \ \quad \ \lambda\in\mathbb{R};
\end{equation}
we obtain a new polynomial vector field $Y$ such that the upper main vertex of its Newton polytope $\mathcal{P}_{Y}$ is contained in $\{-1\}\times\mathbb{N}$ or $\{0\}\times\mathbb{N}$. Moreover, its upper main segment has negative slope.
\end{lemma}
\begin{proof}

Equations \eqref{eq-change-coord-1} give
$$\displaystyle\frac{\partial}{\partial x} = \displaystyle\frac{\partial}{\partial \widetilde{x}}, \ \quad \ \displaystyle\frac{\partial}{\partial y} = -\lambda\displaystyle\frac{\partial}{\partial \widetilde{x}} + \displaystyle\frac{\partial}{\partial \widetilde{y}},$$
which implies that in this new coordinate system we obtain the polynomial vector field (dropping the tildes in order to simplify the notation)
\begin{equation*}
\begin{split}
& Y(x,y) = \\
& = \sum_{(m,n)\in\mathcal{P}}\Bigg{(}(x+\lambda y)^{m}y^{n}\Bigg{)}\Bigg{(}a_{m,n}(x+\lambda y)\displaystyle\frac{\partial}{\partial x} -\lambda b_{m,n}y\displaystyle\frac{\partial}{\partial x} + b_{m,n}y\displaystyle\frac{\partial}{\partial y}\Bigg{)} \\
 & = \sum_{(m,n)\in\mathcal{P}}\Bigg{(}\displaystyle\sum_{i = 0}^{m}\Upsilon_{i}x^{i}y^{m+n-i}\Bigg{)}\Bigg{(}\Big{(}a_{m,n}x+ \Lambda_{m,n}y\Big{)}\displaystyle\frac{\partial}{\partial x} + b_{m,n}y\displaystyle\frac{\partial}{\partial y}\Bigg{)}.
\end{split}
\end{equation*}
in which $\Upsilon_{i} = \binom{m}{i}\lambda^{m-i}$ and $\Lambda_{m,n} = \lambda(a_{m,n}-b_{m,n})$. Geometrically, this change of coordinates add new points to the support in the following way: if $(m,n)\in\mathcal{S}_{X}$, then (depending on the coefficients $a_{m,n}$ and $b_{m,n}$) the new support $\mathcal{S}_{Y}$ of $Y$ will contain the points: $(m,n)$, $(m-1,n+1)$, $\dots$, $(0,m+n)$ and $(-1,m+n+1)$. See Figure \ref{fig-coord-change}.

\begin{figure}[h]\center{
\begin{overpic}[width=1\textwidth]{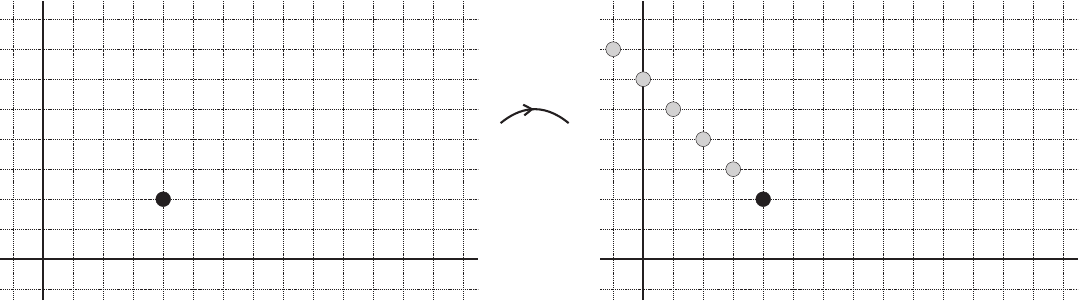}
\end{overpic}}
\caption{\footnotesize{Effect of the coordinate change \eqref{eq-change-coord-1} on the points of the Support $\mathcal{S}$.}}
\label{fig-coord-change}
\end{figure}

Denote by $p_{l} = (m_{l},n_{l})\in\mathcal{P}_{X}$ a point of $\mathcal{P}_{X}$ satisfying the property $m_{l} + n_{l} \geq m + n$, for all $p = (m,n)\in\mathcal{S}_{X}$. Such a point indeed exist because $\mathcal{S}_{X}$ is finite. Then there is two cases to consider. If $\Lambda_{m_{l},n_{l}}\neq 0$, then the new support $\mathcal{S}_{Y}$ contains a point of the form $ q_{l} = (-1,m_{l} + n_{l} +1)$. On the other hand, if $\Lambda_{m_{l},n_{l}} = 0$, then $a_{m_{l},n_{l}} \neq 0$ or $b_{m_{l},n_{l}} \neq 0$. In both cases, the new support $\mathcal{S}_{Y}$ will contain a point of the form $q_{l} = (0,m_{l} + n_{l})$, and $q_{l}$ will be the upper main vertex, since the ordinate of all points of the support $\mathcal{S}_{Y}$ will be strictly lesser than $m_{l} + n_{l}$. Therefore, in this system of coordinates the upper main vertex of $\mathcal{P}_{Y}$ is contained in $\{-1\}\times\mathbb{N}$ or $\{0\}\times\mathbb{N}$. After this construction, it is straightforward to see that in this coordinate system the upper main segment of $\mathcal{P}_{Y}$ has negative slope.
\end{proof}

Although one can always assume that $\mathcal{P}$ is favorable, after the change of coordinates of Lemma \ref{lemma-preparacao-poligono}, $X^{U}_{\Delta}$ may be degenerated in the sense of Definition \ref{def-upper-degenerated}, as it is shown in the next example.  

\begin{example}
Consider the vector field
$$X(x,y) = y^{2}(1+x^{3})\displaystyle\frac{\partial}{\partial x} + x^{2}(1+y^{3})\displaystyle\frac{\partial}{\partial y};$$
and it can be checked that $X^{U}_{\Delta}$ is non degenerated. After a change of coordinates given by \eqref{eq-change-coord-1}, one obtains the vector field
\begin{equation*}
\begin{split}
Y(x,y) & = \Big{(}u^{3}v^{2} - \lambda u^{2} + 2\lambda  u^{2}v^{3} + \lambda^2 uv^{4} - 2\lambda^{2}uv - \lambda^{3}v^{2} + v^{2}\Big{)}\displaystyle\frac{\partial}{\partial x} \\
& + \Big{(} u^{2}v^{3} + u^{2} + 2\lambda uv^{4} + 2\lambda  uv + \lambda^{2}v^{5} + \lambda^{2}v^{2} \Big{)}\displaystyle\frac{\partial}{\partial y};
\end{split}
\end{equation*}
and the Newton polytope $\mathcal{P}_{Y}$ of $Y$ is favorable. However, observe that
$$
Y_{\gamma_{h}}(x,y) = y^{2}(x + \lambda y)^{2}\Big{(}x\displaystyle\frac{\partial}{\partial x} + y\displaystyle\frac{\partial}{\partial y}\Big{)};
$$
in which $Y_{\gamma_{h}}$ denotes the vector field restricted to the monomials associated to the upper main segment $\gamma_{h}$. Observe that $Y_{\gamma_{h}}$ has singularities outside the coordinate axes, and therefore $Y^{U}_{\Delta}$ is degenerated.
\end{example}

\section{Acknowledgements}

Thais Dalbelo and Regilene Oliveira are  supported by Sao Paulo Research Foundation (FAPESP) grant ``Projeto Tem\'atico" 2019/21181-0. Thais M. Dalbelo is supported by Sao Paulo Research Foundation (FAPESP) grant 2023/01018-2. Regilene Oliveira is supported by CNPq grant ``Bolsa de Produtividade em Pesquisa" 304766/2019-4. Otavio Perez is supported by Sao Paulo Research Foundation (FAPESP) grant 2021/10198-9.


\begin{thebibliography}{99}

\bibitem{AlonsoGonzalez} C. Alonso--Gonzalez. \emph{Infinitesimal Hartman--Grobman theorem in dimension three}. \textbf{An. Acad. Bras. Cienc.}  87(3) (2015), 1499--1503.

\bibitem{Berezovskaya} F.M. Berezovskaya. \emph{Topological normal form for a system of two differential equations}. \textbf{Russian Math. Surv.} 33(2) (1978), 227--228.

\bibitem{Berezovskaya2} F.M. Berezovskaya. \emph{Asymptotics of orbits of a Kolmogorov type planar vector field with a fixed Newton polygon}. \textbf{Proc. Amer. Math. Society.} 142(8) (2014), 2671--2681.

\bibitem{Brasselet} J.P. Brasselet. \emph{Introduction to toric varieties}. Publicações Matemáticas, IMPA, Rio de Janeiro (2008).

\bibitem{BrunellaMiari} M. Brunella, M. Miari. \emph{Topological equivalence of a plane vector field with its principal part defined through Newton Polyhedra}. \textbf{J. Diff. Equations} 85 (1990), 338--366.

\bibitem{Bruno} A.D. Bruno. \emph{Power Geometry in Algebraic and Differential Equations}. North-Holland Mathematical Library, 57. Amsterdam (2000).

\bibitem{Cano} J. Cano. \emph{The Newton polygon method for differential equations}, in Computer Algebra and Geometric Algebra with Applications, H. Li, P.J. Olver, G. Sommer, (eds.), Springer, Berlin, 2005, 18--30.

\bibitem{ChiconeSoto} C. Chicone, J. Sotomayor. \emph{On a class of polynomial vector fields in the plane}. \textbf{J. Diff. Equations} 61 (1986), 398--418.

\bibitem{CollDumortierProhens} B. Coll, F. Dumortier, R. Prohens. \emph{Configurations of limit cycles in Liénard equations}. \textbf{J. Diff. Equations} 255 (2013), 4169--4184.

\bibitem{CoxLittleSchenck} D.A. Cox, J.B. Little, H.K. Schenck. \emph{Toric varieties}. Graduate Studies in Mathematics, vol 124, Amer. Math. Soc., Providence (2011). 

\bibitem{DeminaGineValls} M.V. Demina, J. Giné, C. Valls. \emph{Puiseux integrability of differential equations}. \textbf{Qual. Theory Dyn. Sys.} 21(35) (2022).

\bibitem{Dumortier77} F. Dumortier. \emph{Singularities of vector fields on the plane}. \textbf{J. Diff. Equations} 23 (1977), 53–106.

\bibitem{Dumortier} F. Dumortier. \emph{Compactification and desingularization of spaces of polynomial Liénard equations}. \textbf{J. Diff. Equations} 224 (2006), 296--313.

\bibitem{DumortierHerssens} F. Dumortier, C. Herssens. \emph{Polynomial Liénard equations near infinity}. \textbf{J. Diff. Equations} 153 (1999), 1--29.

\bibitem{DumortierLi} F. Dumortier, C. Li. \emph{Quadratic Liénard equations with quadratic damping}. \textbf{J. Diff. Equations} 139 (1997), 41--59.

\bibitem{DumortierLlibreArtes} F. Dumortier, J. Llibre, J.C. Artés. \emph{Qualitative theory of planar differential systems}. Universitext, Springer-Verlag Berlin Heidelberg (2006).

\bibitem{DumortierRousseau} F. Dumortier, C. Rousseau. \emph{Cubic Liénard equations with linear damping}. \textbf{Nonlinearity} 3 (1990), 1015--1039.

\bibitem{GutierrezOliveiraTeixeira} C. Gutierrez, R. Oliveira, M.A. Teixeira. \emph{Positive quadratic differential forms: topological equivalence through Newton polyhedra}. \textbf{J. Dyn. Control Syst.} 12(4) (2006), 489--516.


\bibitem{Kappos} E. Kappos. \emph{Dynamics of polynomial systems at infinity}. \textbf{Eletronic J. Diff. Equations} 22 (2001), 1--15.

\bibitem{Khovanskii1} A.G. Khovanskii. \emph{Newton polyhedra and toroidal varieties}. \textbf{Funkcional. Anal. i Priložen.} 11(4) (1977), 56--64.

\bibitem{Khovanskii2} A.G. Khovanskii. \emph{Newton polyhedra and the genus of complete intersections}. \textbf{Funkcional. Anal. i Priložen.} 12(1) (1978), 51--61.

\bibitem{LiangHuangZhao} H. Liang, J. Huang, Y. Zhao. \emph{Classification of global phase portraits of planar quartic quasi-homogeneous polynomial differential systems}. \textbf{Nonlinear Dyn.} 78 (2014), 1659--1681.

\bibitem{LimaLlibre} M.F.S. Lima, J. Llibre. \emph{Global dynamics of the Benoît system}. \textbf{Annali di Matematica} 193 (2014), 1103--1122.

\bibitem{Panazzolo} D. Panazzolo. \emph{Resolution of singularities of real-analytic vector fields in dimension three}. \textbf{Acta Math} 197(2) (2006), 167--289.

\bibitem{Pelletier} M. Pelletier. \emph{\'Eclatements quasi homog\`enes}. \textbf{Ann. Fac. Sci. Toulouse Math} 4 (1995), 879--937.

\bibitem{PerezSilva2} O.H. Perez, P.R. Silva. \emph{Singular impasse points of planar constrained differential systems}. \textbf{Bull. Belgian Math. Soc. -- Simon Stevin} 29 (2022), 611--643.

\bibitem{Perko} L. Perko. \emph{Differential equations and dynamical systems}. 3rd ed, Texts in Applied Mathematics, Springer-Verlag New York (2001).

\bibitem{Sottile} F. Sottile. \emph{Toric ideals, real toric varieties, and the algebraic moment map}. Topics in Algebraic Geometry and Geometric Modeling, \textbf{Contemp. Math.}, 334, 2003.

\bibitem{Zupanovic2} V. Županović. \emph{Topological equivalence of planar vector fields and their generalized principal part}. \textbf{J. Diff. Equations} 167 (2000), 1--15.

\end{thebibliography}
\end{document}